\documentclass[11pt,a4paper,leqno]{amsart}

\usepackage{tikz,xifthen}
\usepackage{latexsym,verbatim}
\usepackage{amsmath,mathrsfs}
\usepackage{amssymb}
\usepackage{amsthm}
\usepackage{amsfonts}
\usepackage{mathrsfs}
\usepackage{calc}
\usepackage{cite}
\usepackage{color}
\usepackage{graphicx}
\usepackage{enumerate}

\usepackage[english]{babel}
\usepackage[utf8]{inputenc}

\usepackage{hyperref}
\hypersetup{pdfstartview={XYZ null null 1.00}}
        
\theoremstyle{definition}
\newtheorem{theo}{Theorem}[section]
\newtheorem{pipi}[theo]{Proposition}
\newtheorem{lem}[theo]{Lemma}

\newtheorem*{deff*}{Definition}
\newtheorem*{theo*}{Theorem}
\newtheorem{ass}{Assumptions}

\newcommand \jl{{\mathbf{j}_{\geq 2}}}
\newcommand{\pt}[1]{\left( #1 \right)}

\newcommand \erre{{\mathbb{R}}}

\newcommand \enne{{\mathbb{N}}}

\newcommand \RR{{\mathbb{R}}}
\newcommand \NN{{\mathbb{N}}}

\newcommand{\Gcl}[2]{G^{#1}_{\textnormal{cl}} (#2)        }
\newcommand{\Lcl}[2]{  L^{ #1}_{\textnormal{cl}} (#2)    }

\newcommand{\molt}[1]{ \textbf{mult}\left( #1 \right)}

\newcommand{\ptg}[1]{\left\{ #1 \right\}}
\newcommand{\abs}[1]{\left\lvert #1\right\rvert}

\newcommand{\Feyn}[1]{#1\kern-0.45em/}

\begin{document}
%%%%%%%%%%%%%%%%%%%%%%%%%%%%%%%%%%%%%%%%%%%%%%%%%%
%%%%%%%%%%%%%%%%%%%%%%%%%%%%%%%%%%%%%%%%%%%%%%%%%%
\title[Sharp Weyl estimates for tensor products of $\psi$dos]
{Sharp Weyl estimates for tensor products\\ 
of pseudodifferential operators}

\author{Ubertino Battisti}
\address{Università di Torino, Dipartimento di Matematica `Giuseppe Peano', Torino (Italy)}
\email{ubertino.battisti@unito.it}

\author{Massimo Borsero}
\address{Università di Torino, Dipartimento di Matematica `Giuseppe Peano', Torino (Italy)}
\email{massimo.borsero@unito.it}

\author{Sandro Coriasco}
\address{Università di Torino, Dipartimento di Matematica `Giuseppe Peano', Torino (Italy)}
\email{sandro.coriasco@unito.it}

%\today

\begin{abstract}
	We study the asymptotic behavior of the counting function of tensor products 
	of operators, in the cases where the factors are either pseudodifferential operators on closed manifolds, 
	or pseudodifferential operators of Shubin type on $\erre^n$, respectively. 
	We obtain, in particular, the sharpness of the remainder term in the corresponding Weyl formulae,
	which we prove by means of the analysis of some explicit examples.  
\end{abstract}

\maketitle

\tableofcontents

%%%%%%%%%%%%%%%%%%%%%%%%%%%%%%%%%%%%%%%%%%%%%%%%%%
\section*{Introduction}\label{sec:intro}
  Let $P$ be a positive self-adjoint operator of order $m>0$ with domain $H^m(M)\hookrightarrow L^2(M)$,
  $M$ a Riemannian, $n$-dimensional smooth closed manifold. Assume that the resolvent of $P$ is compact, 
  so that the spectrum is discrete and given by a sequence of eigenvalues 
  with finite multiplicities. Let $\{\lambda_j\}_{j \in \NN}=\sigma(P)$ be the set of 
  the eigenvalues of $P$, repeated according to their multiplicity. The counting function $N_P(\tau)$ is defined as 
  \begin{align} 
    \label{eq:weylintro}
    N_P(\tau)=\sum_{\lambda_j \in \sigma(P) \cap [0, \tau)}1= \sum_{\lambda_j <\tau} 1.
  \end{align}
  The Weyl law, see, e.g., \cite{HO68,HO04}, describes the asymptotic expansion of the counting function $N_P(\tau)$, 
  as $\tau$ goes to infinity. It is well known that that the leading
  term of the asymptotic expansion of \eqref{eq:weylintro} depends on the dimension of the 
  manifold, on the order of the operator and
  on its principal symbol, see, e.g., \cite{HO04}. Similar formulae can be obtained in many other different settings,
  see \cite{SV97} and \cite{ANS09} for a detailed analysis and several developments.
  To mention a few specific situations, see \cite{SH87,HL81} for the case of the Shubin calculus on $\RR^n$,
  \cite{BN03} for the anisotropic Shubin calculus, \cite{BC11,CM13,NI03} for the $SG$-operators on $\RR^n$ and
  the manifolds with ends, \cite{LO02} for operators on conic manifolds, \cite{MO08} for operators on cusp manifolds,
  \cite{DD13} for operators on asymptotic hyperbolic manifolds, \cite{BA12,BGPR13} for bisingular operators.
  
  In this paper we study the counting function of the tensor product of $r$ pseudodifferential
  operators. We consider the cases of 
  H\"ormander operators on closed manifolds and of the Shubin calculus on $\RR^n$.
  In the case $r=2$, for classical H\"ormander operators on closed manifolds, the operators 
  we consider are a subclass of the so-called \emph{bisingular operators}, 
  studied by L. Rodino in \cite{RO75} (see also \cite{NR06})
  in connection with the multiplicative property of the Atiyah-Singer index \cite{AS68}.
  An asymptotic expansion of the counting function of bisingular operators was obtained by the first author in 
  \cite{BA12}. The basic tool was the spectral $\zeta$-function, in the spirit of Guillemin's so-called
  \emph{soft proof} of the Weyl law \cite{GU85}. This method allows to determine the leading
  term of the asymptotic expansion in the non-symmetric case (corresponding to a simple first pole of the 
  spectral $\zeta$-function). In the symmetric case the spectral $\zeta$-function 
  has a first pole of order
  $2$. Using a theorem due to Aramaki \cite{AR88}, it has been possible to determine the leading term, which has 
  a behavior of type $\tau^p \log \tau$, as well as the second term, 
  which has a behavior of type $\tau^p$, $p$ being
  the first pole of the spectral $\zeta$-function.
  However, it was not possible, through the aforementioned method, to give a good estimate of the remainder term.
  We notice that the asymptotic behavior of the counting function in the bisingular case 
  has some similarities with the Weyl law in the setting of 
  $SG$-classical operators on manifolds with ends \cite{BC11,CM13}.

  A version of bisingular operators, based on Shubin
  pseudodifferential calculus on $\RR^n$, was introduced in \cite{BGPR13}. 
  The counting function was studied also in this setting, obtaining 
  results analogous to those which hold for the ``standard'' bisingular calculus.
     
  In this paper we consider the same class of operators studied in \cite{GPRV14}, namely,
  tensor products of $r$ pseudodifferential operators, that is
  \[
    A= A_1 \otimes \ldots \otimes A_r.
  \]
  In the sequel we will assume either that each $A_j$ is a classical H\"ormander pseudodifferential operator
  on a $n_j$-dimensional closed manifolds $M_j$, that is $A_j \in \Lcl {m_j} {M_j}$, $j=1,\dots,r$,
  or that each $A_j$ belongs to a classical global Shubin class on $\RR^{n_j}$, that is, $A_j \in \Gcl{m_j} {\RR^{n_j}}$,
  $j=1, \ldots, r$. We also assume that $A$ is positive, self-adjoint
  and Fredholm. It is straightforward to check that the Fredholm property of $A$ implies that $A_j$
  is invertible for any $j=1, \ldots , r$. 
  We illustrate here our results in the case $r=2$, see Section \ref{sec:prodn} below for the statements 
  which hold for an arbitrary number of factors.
  
  Denoting by $\sigma\pt{A_1}=\ptg{\lambda_j}_{j\in \NN}$ and 
  $\sigma\pt{A_2}=\ptg{\mu_k}_{k \in \NN}$
  the spectra of $A_1$ and $A_2$, with eigenvalues repeated according with their multiplicities, we easily obtain 
  that the spectrum of $A$ is given by
  \[
    \sigma(A)=\ptg{\lambda_j\cdot \mu_k}_{(j,k) \in \NN^2}.
  \]
  Therefore
  \begin{align}
    \label{eq:countA}
    N_{A}(\tau)=\sum_{\rho \in \sigma\pt{A}\cap [0, \tau)} 1= 
    \sum_{\lambda_j \cdot \mu_k <\tau} 1.
  \end{align}
  
  Assume that $A=A_1\otimes A_2$ is  positive, self-adjoint and Fredholm,
  with $A_1 \in \Lcl{m_1}{M_1}$, $A_2 \in \Lcl{m_2}{M_2}$, $m_1,m_2>0$,
  $\dim M_1=n_1$, $\dim M_2=n_2$, and
  $\frac{n_1}{m_1}>\frac{n_2}{m_2}$.  
  Our first main result, proved in Theorem \ref{thm:teoprin}, states that, under such assumptions,
  \begin{equation}
  \label{eq:asintshar}
    N_A(\tau) = \left\{
    \begin{array}{lcl}
    \dfrac{C_1}{n_1} \zeta\left(A_2, \dfrac{n_1}{m _1}\right) \tau^{\frac{n_1}{m_1}} + \mathcal{O}\left(\tau^{\frac{n_1-1}{m_1}}\right)  & \mbox{if} & \dfrac{n_2}{m_2} < \dfrac{n_1-1}{m_1}, \\
    \dfrac{C_1}{n_1}  \zeta\left(A_2, \dfrac{n_1}{m _1}\right) \tau^{\frac{n_1}{m_1}} + \mathcal{O}\left(\tau^{\frac{n_1-1}{m_1}} \log \tau\right) & \mbox{if} & \dfrac{n_2}{m_2} = \dfrac{n_1-1}{m_1},\\
     \dfrac{C_1}{n_1}  \zeta\left(A_2, \dfrac{n_1}{m _1}\right) \tau^{\frac{n_1}{m_1}} + \mathcal{O}\left(\tau^{\frac{n_2}{m_2}}\right)  & \mbox{if} & \dfrac{n_2}{m_2} > \dfrac{n_1-1}{m_1},
    \end{array} \right.
  \end{equation}
  for $\tau \rightarrow + \infty$. In \eqref{eq:asintshar}, $\zeta$ denotes the spectral $\zeta$-function and 
  \begin{equation*}
    C_1 = \dfrac{1}{(2\pi)^{n_1}} \int_{M_1}\int_{\mathbb{S}^{n_1-1}}
    \frac{d\theta_1 dx_1}{[a_{m_1}(x_1,\theta_1)]^{\frac{n_1}{m_1} }}  . 
  \end{equation*} 
  A similar statement holds for the tensor product of two Shubin operators with positive order.
  Moreover, using spherical harmonics, we show that the estimate \eqref{eq:asintshar}
  is sharp. 
  
  In \cite{GPRV14}, Gramchev, Pilipovi\'c, Rodino and Vindas considered the same class of operators, finding
  a slightly weaker estimate for the remainder term of the Weyl formula. Explicitely, they prove that, under the assumptions 
  stated above, 
  \[
    N_A(\tau)= \dfrac{C_1}{n_1} \zeta\left(A_2, \dfrac{n_1}{m _1}\right) \tau^{\frac{n_1}{m_1}}+ 
    \mathcal{O}(\tau^\delta)
  \]
  where $\max \ptg{\frac{n_1-1}{m_1}, \frac{n_2}{m_2}}<\delta<\frac{n_1}{m_1}$.

  The asymptotic expansion in \eqref{eq:asintshar} is related with the position
  of the first poles of the spectral $\zeta$-function associated with $A_1$ and $A_2$, 
  as sketched in the following pictures.
  
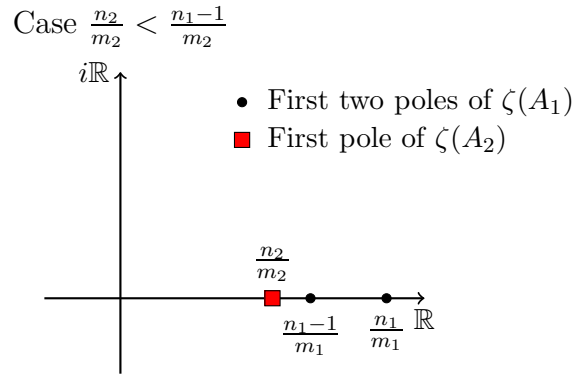
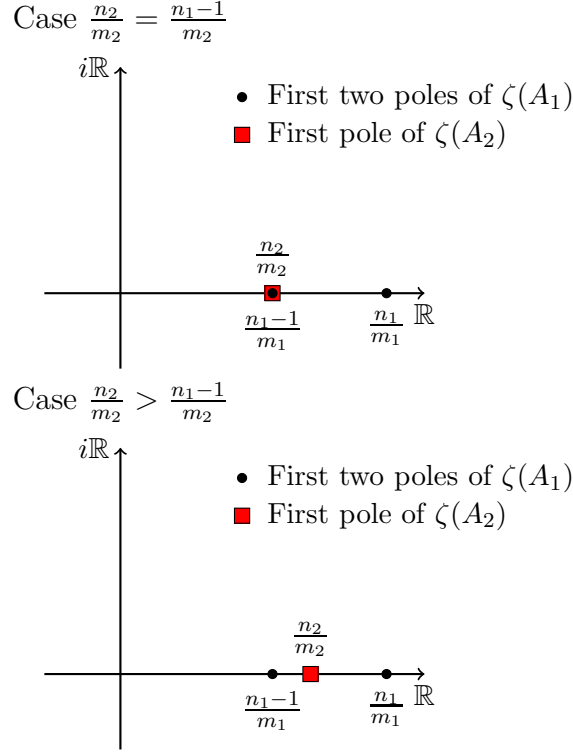
\begin{center}

\begin{tikzpicture}
 
   \draw[thick,->] (-1,0) -- (4,0) node[below] {$\mathbb{R}$};
   \draw[thick, ->] (0,-1) -- (0,3) node[left] {$i\mathbb{R}$};
   \draw [fill] (3.5,0) circle (0.06);
   \node[below] at (3.5,-0.1) {$\frac{n_1}{m_1}$};
   \draw [fill] (2.5,0) circle (0.06);
   \node[below] at (2.5,-0.1) {$\frac{n_1-1}{m_1}$};
   \draw [fill=red] (1.9,-0.1) rectangle (2.1, 0.1);
   \node[above] at (2, 0.1) {$\frac{n_2}{m_2}$};
   \node[right] at (1.8,2.6) {First two poles of $\zeta(A_1)$};
   \draw [fill] (1.6, 2.6)circle (0.06);
   \node[right] at (1.8,2.1) {First pole of $\zeta(A_2)$};
   \draw [fill=red] (1.5, 2.) rectangle  (1.7, 2.2);
   \node[above] at (0,3.2) {\large{Case $\frac{n_2}{m_2}<\frac{n_1-1}{m_2}$}};
\end{tikzpicture}

\begin{tikzpicture}
 
   \draw[thick,->] (-1,0) -- (4,0) node[below] {$\mathbb{R}$};
   \draw[thick, ->] (0,-1) -- (0,3) node[left] {$i\mathbb{R}$};
   \draw [fill] (3.5,0) circle (0.06);
   
   \node[below] at (2,-0.1) {$\frac{n_1-1}{m_1}$};
   \draw [fill=red] (1.9,-0.1) rectangle (2.1, 0.1);
   \node[below] at (3.5,-0.1) {$\frac{n_1}{m_1}$};
   \draw [fill] (2,0) circle (0.06);
   \node[above] at (2, 0.1) {$\frac{n_2}{m_2}$};
   \node[right] at (1.8,2.6) {First two poles of $\zeta(A_1)$};
   \draw [fill] (1.6, 2.6)circle (0.06);
   \node[right] at (1.8,2.1) {First pole of $\zeta(A_2)$};
   \draw [fill=red] (1.5, 2.) rectangle  (1.7, 2.2);
   \node[above] at (0,3.2) {\large{Case $\frac{n_2}{m_2}=\frac{n_1-1}{m_2}$}};
\end{tikzpicture}
  
\begin{tikzpicture}
 
   \draw[thick,->] (-1,0) -- (4,0) node[below] {$\mathbb{R}$};
   \draw[thick, ->] (0,-1) -- (0,3) node[left] {$i\mathbb{R}$};
   \draw [fill] (3.5,0) circle (0.06);
   \node[below] at (3.5,-0.1) {$\frac{n_1}{m_1}$};
   \draw [fill] (2,0) circle (0.06);
   \node[below] at (2,-0.1) {$\frac{n_1-1}{m_1}$};
   \draw [fill=red] (2.4,-0.1) rectangle (2.6, 0.1);
   \node[above] at (2.5, 0.1) {$\frac{n_2}{m_2}$};
   \node[right] at (1.8,2.6) {First two poles of $\zeta(A_1)$};
   \draw [fill] (1.6, 2.6)circle (0.06);
   \node[right] at (1.8,2.1) {First pole of $\zeta(A_2)$};
   \draw [fill=red] (1.5, 2.)rectangle  (1.7, 2.2);
   \node[above] at (0,3.2) {\large{Case $\frac{n_2}{m_2}>\frac{n_1-1}{m_2}$}};
\end{tikzpicture}

\end{center}

  The key point in the proof of our results is the following equivalence, explained in \eqref{eq:sommu}:
  \[
    N_A\pt{\tau}= \sum_{\lambda_j \cdot \mu_k<\tau}1
    = \sum_{\mu_k<\tau} N_{A_2}\pt{\frac{\tau}{\mu_k}}.
  \]
  The argument is then a careful application of the well known sharp Weyl law.
  A main aspect is the possibility to estimate the reminder term, in the Weyl law
  of $A_2$ evaluated in $\frac{\tau}{\mu_k}$, uniformly with respect to $\mu_k$.
  
  The paper is organized as follows. In Section \ref{sec:taub}, we shortly recall the
  Weyl laws in the case of the H\"ormander calculus on closed manifolds
  and of the Shubin calculus on $\RR^n$. We also study the asymptotic behavior of
  the sum 
  \[
   \sum_{\mu_k< \tau} \frac{1}{\mu_k^c}
  \]
  for different ranges of $c\in\RR$, where $\ptg{\mu_k}_{k\in \NN}$ is the spectrum of an operator in the calculus
  we consider. In Section \ref{sec:asy}, we prove our main results in the case of tensor products 
  of two factors.
  In Section \ref{sec:prodn}, we extend the results to the case of tensor products of $r>2$ factors. 
  In Section \ref{sec:contro},
  we show that our estimates of the remainder term of the Weyl law are sharp, focusing again on the case
  of tensor products of two factors. Finally, we collect in the Appendix some remarks concerning
  the connection of this analysis with lattice problems, in particular with
  the Dirichlet divisor problem in the classic setting and in the anisotropic case.
  
\section*{Acknowledgements}
We wish to thank L. Rodino, S. Pilipovi\'c, F. Nicola, and J. Seiler, for useful discussions
and comments.
The first author has been supported by the Gruppo Nazionale per l'Analisi
Matematica, la Probabilit\`a e le loro Applicazioni (GNAMPA) of the Istituto
Nazionale di Alta Matematica (INdAM).
The second and third author have been partially supported by the Gruppo 
Nazionale per l'Analisi Matematica, la Probabilit\`a e le loro Applicazioni 
(GNAMPA) of the Istituto Nazionale di Alta Matematica (INdAM).

%   
%%%%%%%%%%%%%%%%%%%%%%%%%%%%%%%%%%%%%%%%%%%%%%%%%%%%%%%%%%%%%
\section{Preliminary Results}\label{sec:taub}
%%%%%%%%

We recall well known results on the sharp Weyl law in the case of operators on closed manifolds and
of operators of Shubin type on $\RR^n$, see, e.g., Hormander \cite{HO68},
Hellfer, Robert \cite{HL81}, see also \cite{He84}.
\begin{theo}[Sharp Weyl law]
 Let $A$ be a positive self-adjoint elliptic classical
 pseudodifferential operator in $\Lcl{m}{M}$, with $M$ a closed manifold of dimension
 $n$, and let $\sigma\pt{A}=\ptg{\lambda_j}_{j \in \NN}$ be its spectrum. Then,
 \begin{equation}
   \label{eq:weylcpt}
   N_A(\lambda)= \sum_{\lambda_j < \lambda} 1= \frac{C_A}{n} \lambda^{\frac{n}{m}}+R_A\pt{\lambda},
 \end{equation}
where
\[
  C_A= \frac{1}{\pt{2 \pi}^{n}}\int_{M}\int_{\mathbb{S}^{n-1}}\frac{d\theta dx}{[a_m\pt{x, \theta}]^{\frac{n}{m}}}, 
\]
with $a_m$ the principal homogeneous symbol of $A$,
and
\begin{align}
  \nonumber
  &\limsup_{\lambda \to +\infty} \frac{\abs{N_{A}\pt{\lambda} - \frac{C_A}{n} \lambda^{\frac{n}{m}}}}
  {\lambda^{\frac{n-1}{m}}}\\
  \label{eq:restocpt}
  &=\limsup_{\lambda \to + \infty} \frac{\abs{R_A\pt{\lambda}}}{\lambda^{\frac{n-1}{m}}}
  <+\infty.
\end{align}
Analogously, let $P\in \Gcl{m}{\RR^{n}}$ be a positive self-adjoint elliptic classical
pseudodifferential operator of Shubin type on $\RR^n$ with $m>0$, and let $\sigma\pt{P}=\ptg{\mu_k}_{k \in \NN}$
be its spectrum. Then,
\begin{equation}
  \label{eq:weylshubin}
  N_P\pt{\lambda}= \sum_{\mu_k<\lambda}1= \frac{K_P}{2n} \lambda^{\frac{2n}{m}}+ R_P\pt{\lambda},
\end{equation}
where
\[
  K_P= \frac{1}{\pt{2\pi}^{2n}} \int_{\mathbb{S}^{2n-1}}\frac{ d\theta}{[p_{m}\pt{\theta}]^{\frac{2n}{m}}},
\]
with $p_m$ the principal homogeneous symbol of $P$,
and
\[
\limsup_{\lambda \to + \infty} \frac{\abs{N_{P}\pt{\lambda} - \frac{K_P}{2n} \lambda^{\frac{2n}{m}}}}
  {\lambda^{\frac{2n-1}{m}}}
  =\limsup_{\lambda \to +\infty} \frac{\abs{R_P\pt{\lambda}}}{\lambda^{\frac{2n-1}{m}}}
  <+\infty.
\]
\end{theo}

The next Propositions \ref{pipi:asin} and \ref{pipi:asinSh}
will be crucial in our proof of the Weyl law with sharp remainder for tensor products.
They follow as consequence of well known properties of the spectra of positive self-adjoint operators. 
We examine in detail only the case of H\"ormander pseudodifferential operators on closed manifold, 
since the argument for the case of Shubin operators is similar.

\begin{pipi} \label{pipi:asin}
Let $M$ be a closed manifold of dimension $n$, and $A \in \Lcl{m}{M}$, $m > 0$, be 
elliptic, positive and self-adjoint,
with spectrum $\sigma(A) = \{\mu_k\}_{k \in \enne}$. Define
  \begin{equation}
    \label{eq:Asint}
    F_A(\tau,c)=\sum_{\mu_k < \tau} \frac{1}{\mu_k^c} = \left\{
    \begin{array}{lcl}
      F_1\pt{\tau}  & \mbox{if} & c > \dfrac{n}{m}, \\
      F_2\pt{\tau}  & \mbox{if} & c = \dfrac{n}{m},\rule{0mm}{7mm} \\
      F_3\pt{\tau}  & \mbox{if} & c < \dfrac{n}{m}.\rule{0mm}{7mm}
    \end{array} \right.
  \end{equation}
Then, 
\[
  \limsup_{\tau \to +\infty} \frac{\zeta(A,c)-F_1\pt{\tau}}{\tau^{\frac{n}{m}-c}}= \kappa_1,\,
  \limsup_{\tau \to +\infty} \frac{F_2\pt{\tau}}{\log \tau} = \kappa_2,\,
  \limsup_{\tau \to +\infty} \frac{F_3\pt{\tau}}{\tau^{\frac{n}{m}-c}} = \kappa_3,
\]
for suitable positive constants $\kappa_1$, $\kappa_2$, $\kappa_3$. That is,  for $\tau\to +\infty$,
\[
\zeta(A,c)-F_1(\tau) = \mathcal{O}\pt{\tau^{\frac{n}{m}-c}}, \; F_2(\tau) = \mathcal{O}\pt{\log \tau}, \;
F_3(\tau) = \mathcal{O}\pt{\tau^{\frac{n}{m}-c}}.
\]
\end{pipi}
\begin{proof}
  If $c>\frac{n}{m}$ it is immediate that the series $\sum_{k=0}^{\infty} \frac{1}{\mu_k^c}$ is convergent, 
  in view of the holomorphic properties of the spectral $\zeta$-function associated with $A$.
  To prove the asymptotic properties of $\zeta(A,c)-F_1(\tau)$, we switch to
  $B=A^{1/m}$, so that the order of $B$ is one and $\sigma(B)={\mu_k^{1/m}}$.
  We have
  \begin{align}
  \nonumber
   \zeta(A,c)- F_1(\tau)&= \sum_{\mu_k \geq \tau }\frac{1}{\mu_k^c}
      = \sum_{\mu_k^{1/m}\geq \tau^{1/m} }\frac{1}{\pt{\mu_k^{1/m}}^{c\,m}}\\
   \label{eq:diff}
   &= \int_{\tau^{1/m}}^{+\infty}\frac{1}{\mu^{c\,m}} dN_B(\mu).
  \end{align}
  Since $B$ is of order one, it is well known that
  \begin{align}
    \label{eq:grigis}
    N_B(\lambda+1)- N_B(\lambda)\leq
    \sharp \ptg{\sigma(B) \cap[\lambda, \lambda+1] }= \mathcal{O}(\lambda^{n-1}), \lambda\to+\infty  
  \end{align}
  (see, e.g., \cite[\S~12]{GS94}). Using \eqref{eq:grigis} and the properties of Stieltjes integral, 
  we obtain,
  for $\tau\to+\infty$,
  \begin{align*}
   \zeta(A,c)- F_1(\tau)&= \int_{ \tau^{1/m}}^{+\infty} \frac{1}{\mu^{c \, m}} dN_{B}(\mu)\\
            &\leq \sum_{j=[\tau^{1/m}]-1}^{\infty}  \sup_{\mu\in [j, j+1]} 
	   \pt{\frac{1}{\mu^{c\,m}}} \pt{N_B(j+1)- N_B(j)}\\
            &\leq \kappa \sum_{j=[\tau^{1/m}]-1}^{\infty} \frac{1}{j^{c\,m-n+1}}\\
            &\leq \kappa\int_{[\tau^{1/m}]-1}^{+\infty} \frac{1}{\pt{t-1}^{c\,m-n+1}} dt\\
            &=\kappa\frac{1}{c\,m-n} [\tau^{1/m}-2]^{n-m\,c} \in \mathcal{O}\pt{\tau^{\frac{n}{m}-c}}.
  \end{align*}
  where $[a]$ denotes the minimum integer such that $[a]\geq a$.
  
 To prove the results for $F_2$ and $F_3$
  we can assume, without loss of generality, that $\mu_0=\tilde{\mu}_0=1$.
  Using again the properties of the Stieltjes integral, we write
  \[
    F_A(\tau,c)=\!\int_{1}^{\tau^{1/m}}\hspace{-2mm} \frac{1}{\mu^{c \, m}} dN_{B}(\mu)\leq \sum_{j=1}^{[\tau^{1/m}]} \sup_{\mu\in [j, j+1]} 
    \pt{\frac{1}{\mu^{c\,m}}} \pt{N_B(j+1)- N_B(j)}.
  \]
  Let us initially suppose that $c>0$, so that $\frac{1}{x^c}$ is a decreasing function on $[1,+\infty)$. 
  In view of \eqref{eq:grigis}, we have
  \begin{align}
      \nonumber
      \int_{1}^{\tau^{1/m}} \frac{1}{\mu^{c \, m}} dN_{B}(\mu)&
      \leq \sum_{j=1}^{[\tau^{1/m}]} \frac{1}{j^{c\,m}}\mathcal{O}(j^{n-1})
      \leq \widetilde{\kappa} \sum_{j=1}^{[\tau^{1/m}]} \frac{1}{j^{c\,m-n+1}}\\
							    \label{eq:dec}
							    &\leq \widetilde{\kappa} 
							    \pt{\int_1^{[\tau^{1/m}]} t^{n-c\,m-1}dt+1}.
  \end{align}
  By integration, we find
  \[
   F_A(\tau,c)=\int_{1}^{\tau^{1/m}} \frac{1}{\mu^{c \, m}} dN_{B}(\mu) \leq \left\{
      \begin{array}{lcl}
      \dfrac{\widetilde{\kappa}_1}{n-c\,m} \tau^{\frac{n}{m}-c}   & \mbox{if} & 0<c < \dfrac{n}{m}, \\
      \dfrac{\widetilde{\kappa}_2}{m} \log \tau  & \mbox{if} & c = \dfrac{n}{m},\rule{0mm}{7mm}
    \end{array} \right. 
  \]
  as claimed. Finally, if $c\le0$, then $\frac{1}{\mu^c}$ is a non-decreasing function and also in this case, similarly to 
  \eqref{eq:dec}, we obtain
  \[
    F_A(\tau,c)\leq \kappa \int_1^{[\tau^{1/m}]} \pt{x+1}^{n-c\,m-1}dx\leq 
    \frac{ \widetilde{\kappa}_3} {n- c\,m} \tau^{\frac{n}{m}-c}.
  \]
  The proof is complete.
\end{proof}
\begin{pipi}\label{pipi:asinSh}
Let $P \in \Gcl{m}{\RR^n}$ be an elliptic, positive and self-adjoint Shubin operator of order $m>0$, with
spectrum given by $\sigma(P) = \{\lambda_j\}_{j \in \enne}$. 
Define
  \begin{equation}
    \label{eq:Asintsh}
    F_P(\tau,c)=\sum_{\lambda_j < \tau} \frac{1}{\lambda_j^c} = \left\{
    \begin{array}{lcl}
      F_1\pt{\tau}  & \mbox{if} & c > \dfrac{2n}{m}, \\
      F_2\pt{\tau}  & \mbox{if} & c = \dfrac{2n}{m},\rule{0mm}{7mm} \\
      F_3\pt{\tau}  & \mbox{if} & c < \dfrac{2n}{m}.\rule{0mm}{7mm}
    \end{array} \right. 
  \end{equation}
Then, 
\[
  \limsup_{\tau \to +\infty} \frac{\zeta(P,c)- F_1\pt{\tau}}{\tau^{\frac{2n}{m}-c}} = {\kappa}_1,\,
  \limsup_{\tau \to +\infty} \frac{F_2\pt{\tau}}{\log \tau} = {\kappa}_2,\,
  \limsup_{\tau \to +\infty} \frac{F_3\pt{\tau}}{\tau^{\frac{2n}{m}-c}} = {\kappa}_3,
\]
for suitable positive constants $\kappa_1$, $\kappa_2$, $\kappa_3$. That is, for $\tau\to+\infty$, 
\[
\zeta(P,c)-F_1(\tau) = \mathcal{O}\pt{\tau^{\frac{2n}{m}-c}},\;
F_2(\tau) = \mathcal{O}\pt{\log \tau},\;
F_3(\tau) = \mathcal{O}\pt{\tau^{\frac{2n}{m}-c}}.
\]
\end{pipi}

%%%%%%%%%%%%%%%%%%%%%%%%%%%%%%%%%%%%%%%%%%%%%%%%%%%%%%%%%%%%%

\section{Spectral asymptotics for the tensor product of two operators}\label{sec:asy}
We start considering the case of the tensor product of $2$ operators.
Let $M_1, M_2$ be two compact manifolds of dimension $n_1,n_2$, respectively. 
Let $A= A_1 \otimes A_2$, $A_j \in \Lcl {m_j} {M_j}$, $m_j>0$, $j=1,2$.
Assume that the spectra of $A_1$ and $A_2$ are sequences of eigenvalues, and set
\begin{align*}
	\sigma(A_1) = \{\lambda_j\}_{j \in \enne}, \qquad \sigma(A_2) = \{\mu_k\}_{k \in \enne}, 
\end{align*}
so that
\begin{equation*}
	\sigma(A) = \{\lambda_j\cdot \mu_k : \lambda_j \in \sigma(A_1), \mu_k \in \sigma(A_2)\}.
\end{equation*}
For simplicity, we start with the case $m_1=1$ and $n_1> \frac{n_2}{m_2}$.
Let $c$ be an arbitrary positive constant and $B$ an operator with spectrum
$\sigma(B)=\ptg{\mu_k}_{k\in \NN}$,
then $\sigma\!\pt{c\, B}= \ptg{c\cdot \mu_k}_{k \in \NN}$. There is a simple and useful formula relating 
the counting functions $N_{c\;B}$ and $N_{B}$, namely 
\begin{align}
  \label{eq:prodspectrum}
  N_{c \, B}(\tau)= \sum_{c\cdot \mu_k<\tau}1= \sum_{\mu_k<\frac{\tau}{c}}1= 
  N_{B}\!\pt{\frac{\tau}{c}}.
\end{align}
In particular, \eqref{eq:prodspectrum} implies that,
without loss of generality, we can assume\footnote{In fact,
if that condition were not true, we could consider the operator $c^2 A$, with
$c=(\min\{\lambda_j, \mu_k\} - \varepsilon)^{-1}$, $\varepsilon > 0$ small enough.} 
$\lambda_j > 1$ and $\mu_k>1$ for all $j,k$.
Let us now summarize the hypotheses on the factors $A_1$, $A_2$.
\begin{ass}
 \label{assu1st} 
  \begin{align*}
    &\text{$M_1,M_2$ smooth closed manifolds of dimensions $n_1,n_2$, respectively};
    \\
    &A=A_1\otimes A_2, \; A_1\in \Lcl{1}{M_1}, A_2\in \Lcl{m_2}{M_2}, \quad m_2>0, n_1>\frac{n_2}{m_2};
     \\
    & A_1,A_2 \text{ positive, self-adjoint, elliptic};
    \\
    &\sigma\pt{A_1}=\ptg{\lambda_j}_{j\in \NN}, \, \sigma\pt{A_2}=\ptg{\mu_k}_{k\in \NN},
     \quad \lambda_j>1, \mu_k>1,\text{ for all } j, k.
  \end{align*}
\end{ass}

Since $\lambda_j, \mu_k > 1$ for all $j,k$, using \eqref{eq:prodspectrum}, we 
have\footnote{Recall that $\lambda_j>1$ for all $j$. In the first term 
of \eqref{eq:sommu} we can reduce the summation to $\mu_k<\tau$
since, otherwise, we would have $ \lambda_k \cdot \mu_k\geq \tau$ for all $k$, and the second summation would
be zero.}
\begin{align}
\nonumber
N_{A} (\tau)&= \sum_{\lambda_j \cdot \mu_k<\tau} 1= 
\sum_{\mu_k <\tau}\pt{ \sum_{ \lambda_j \cdot \mu_k< \tau} 1} =\\
\label{eq:sommu}
&= 
\sum_{\mu_k<\tau} N_{\mu_k A_1}(\tau)
= \sum_{\mu_k<\tau} N_{A_1}\pt{\frac{\tau}{\mu_k}}.
\end{align}
\begin{pipi} \label{pipi:iunouno}
Let $A$, $A_1$ and $A_2$ be as in Assumptions \ref{assu1st}. Then,
\begin{align*}
  N_{A} (\tau) = \sum_{\mu_k< \tau } \pt{ 
  \frac{C_1}{n_1} \pt{ \frac{\tau}{\mu_k} }^{n_1} + \frac{1}{\mu_k^{n_1-1} } r_k(\tau) },
\end{align*}
with
\begin{align}
  \label{eq:C1}
  C_1= \frac{1}{\pt{2 \pi}^{n_1}}\int_{M_1}\int_{\mathbb{S}^{n_1-1}}
  \frac{d\theta_1 dx_1}{[a_{m_1}(x_1,\theta_1)]^{\frac{n_1}{m_1}}},
\end{align}
and  $r_k(\tau)$ is $\mathcal{O}\pt{\tau^{n_1-1}}$, uniformly with respect to $\mu_k$. That is,
there exists a positive constant $C$ such that
\begin{align}
  \label{eq:unifbound}
   |r_k(\tau)|\leq C \tau^{n_1-1}, \quad \text{for all } k \in \NN.
\end{align}
\end{pipi}
\begin{proof}
 By \eqref{eq:sommu} we have
 \[
  N_A(\tau)= \sum_{\mu_k <\tau} N_{A_1}\pt{\frac{\tau}{\mu_k}}.
 \]
 Using \eqref{eq:weylcpt}, we can write
 \begin{equation}
  \label{eq:som}
   N_A(\tau)= \sum_{\mu_k<\tau} \pt{\frac{C_1}{n_1} \frac{\tau^{n_1}}{\mu_k^{n_1}} +R_A\pt{\frac{\tau}{\mu_k}}}.
 \end{equation}
 Equation \eqref{eq:restocpt} implies that
 \[
    |R_A(t)|\leq \kappa t^{n_1-1}, \quad t>1, 
 \]
 for a suitable constant $\kappa$. Since $\mu_k<\tau\Rightarrow\frac{\tau}{\mu_k}>1$ in the summation 
 \eqref{eq:som}, we can write
 \[
   \abs{R_A\pt{\frac{\tau}{\mu_k}}}\leq C \pt{\frac{\tau}{\mu_k}}^{n_1-1}.
 \]
 Hence, setting
 \[
  r_k(\tau)=\mu_k^{n_1-1} R_A\pt{\frac{\tau}{\mu_k}},
 \]
 we have the assertion.

\end{proof}
\begin{lem}\label{lem:iuno}
Let $A,A_1,A_2$ be as in Assumptions \ref{assu1st}, and assume $n_1>\dfrac{n_2}{m_2}$. Then
we have, for $\tau \rightarrow +\infty$, 
\begin{equation*}
N_{A}(\tau) = \left\{
\begin{array}{lcl}
\dfrac{C_1}{n_1} \zeta\left(A_2, n_1\right) \tau^{n_1} + \mathcal{O}(\tau^{n_1-1})
& \mbox{if} & \dfrac{n_2}{m_2} < n_1-1, \\
\dfrac{C_1}{n_1} \zeta\left(A_2, n_1\right) \tau^{n_1} + 
\mathcal{O}\left(\tau^{n_1-1}\log \tau\right)
& \mbox{if} & \dfrac{n_2}{m_2} = n_1-1, \rule{0mm}{7mm} \\
\dfrac{C_1}{n_1} \zeta\left(A_2, n_1\right) \tau^{n_1} + 
\mathcal{O}\left(\tau^{\frac{n_2}{m_2}}\right)
& \mbox{if}  & \dfrac{n_2}{m_2} > n_1-1, \rule{0mm}{7mm}
\end{array} \right. 
\end{equation*}
where $C_1$ is given by \eqref{eq:C1}.
\end{lem}
\begin{proof}
Using Proposition \ref{pipi:iunouno} we obtain
\begin{equation*}
N_{A} (\tau) = \sum_{\mu_k< \tau } \pt{ \frac{C_1}{n_1} \pt{ \frac{\tau}{\mu_k}}^{n_1} +
\frac{1} {{\mu_k}^{n_1-1} } r_k \pt{\tau} }, 
\end{equation*}
where $r_k(\tau)$ is uniformly $\mathcal{O}\pt{\tau^{n_1-1}}$ for $\tau\to+\infty$, in the sense of 
\eqref{eq:unifbound}. We can then write
\begin{align}
\nonumber
&\hspace*{-1cm}
\abs{N_A(\tau)- \frac{C_1}{n_1} \zeta(A_2, n_1) \tau^{n_1}} \\
\nonumber
&=\abs{\sum_{\mu_k < \tau}\pt{  \frac{C_1}{n_1} \frac{\tau^{n_1}}{\mu_k^{n_1}} +
\frac{1}{\mu_k^{n_1-1}} r_k (\tau^{n_1-1})}- \frac{C_1}{n_1} \zeta(A_2, n_1) \tau^{n_1} }\\
\nonumber
&\leq 
 \frac{C_1}{n_1} \tau^{n_1}
 \abs{ F_{A_2}(\tau,n_1) - \zeta(A_2, n_1) }+
 \abs{\sum_{\mu_k < \tau} \frac{1}{\mu_k^{n_1-1}} r_k (\tau^{n_1-1})}\\
\label{eq:mainest}
&\leq 
 \frac{C_1}{n_1} \tau^{n_1}
 \abs{ F_{A_2}(\tau,n_1) - \zeta(A_2, n_1) }+
 C\tau^{n_1-1} F_{A_2}(\tau,n_1-1). 
\end{align} 
Let us start with the case $n_1-1 > \frac{n_2}{m}$. Using \eqref{eq:mainest}, we find 
\begin{align*}
&\hspace*{-1cm}\limsup_{\tau \rightarrow +\infty} 
\frac{\abs{N_A(\tau)- \frac{C_1}{n_1} \zeta(A_2, n_1) \tau^{n_1}}}{\tau^{n_1-1}} \\
 \leq &\frac{C_1}{n_1}\limsup_{\tau\to +\infty} \tau \abs{\zeta(A_2, n_1)- F_{A_2}(\tau,n_1)}  +
 C\limsup_{\tau \to +\infty} F_{A_2}(\tau,n_1-1). 
\end{align*}
Since
\[
  n_1>n_1-1> \frac{n_2}{m_2}\Rightarrow \frac{n_2}{m_2}-n_1<-1,
\]
$\zeta(A_2,n_1)-F_1(\tau) = \mathcal{O}\pt{\tau^{\frac{n_2}{m_2} -n_1}}$ 
for $\tau\to+\infty$, in view of Proposition \ref{pipi:asin}. It follows that
\[
 \limsup_{\tau \to +\infty} \tau |\zeta(A_2,n_1)-F_1(\tau)|\leq 
 \widetilde{C} \limsup_{\tau \to +\infty} \tau^{\frac{n_2}{m_2}-n_1+1}=0,
\]
which implies
\[
\limsup_{\tau \rightarrow +\infty} 
\frac{\abs{N_A(\tau)- \frac{C_1}{n_1} \zeta(A_2, n_1) \tau^{n_1}}}{\tau^{n_1-1}}
\leq  C \limsup_{\tau\to+\infty} F_{A_2}(\tau,n_1-1)=C\zeta(A_2, n_1-1).
\]
Since $n_1-1>\frac{n_2}{m_2}$, $\zeta(A_2, n_1-1)$ is finite, and we have the desired assertion.

In the case $n_1-1 = \frac{n_2}{m_2}$, from \eqref{eq:mainest} we analogously get
\begin{align*}
  &\hspace*{-0.4cm}\limsup_{\tau \rightarrow +\infty} 
  \frac{\abs{N_A(\tau)- \frac{C_1}{n_1} \zeta(A_2, n_1)
  \tau^{n_1}}}{\tau^{n_1-1}\, \log \tau} \\
  &\leq \frac{C_1}{n_1}\limsup_{\tau\to +\infty} \frac{\tau}{\log \tau} \abs{\zeta(A_2, n_1)- F_{A_2}(\tau,n_1)}  +
 C \limsup_{\tau \to +\infty}\frac{ 1}{\log \tau} F_{A_2}\!\!\pt{\tau,\frac{n_2}{m_2}}. 
\end{align*}
Since  $n_1>n_1-1=\frac{n_2}{m_2}$, in view of Proposition \ref{pipi:asin} we find
\[
\zeta(A_2,n_1)-F_1(\tau) = \mathcal{O}\pt{\tau^{-1}},
\quad
F_{A_2}\pt{\tau, \frac{n_2}{m_2}}=F_2(\tau)= \mathcal{O}\pt{\log\tau},
\]
so that
\[
  \limsup_{\tau \rightarrow +\infty} 
  \frac{\abs{N_A(\tau)- \frac{C_1}{n_1} \zeta(A_2, n_1) \tau^{n_1}}}{\tau^{n_1-1} \log \tau} 
 \leq  \widetilde{C}, 
\]
as claimed. 

Finally, in the case $n_1-1<\frac{n_2}{m_2}$, \eqref{eq:mainest} gives
\begin{align*}
&\hspace*{-2cm}\limsup_{\tau \rightarrow +\infty} 
\frac{\abs{N_A(\tau)- \frac{C_1}{n_1} \zeta(A_2, n_1) \tau^{n_1}}}{\tau^{\frac{n_2}{m_2}}} \\
 \leq \frac{C_1}{n_1}&\limsup_{\tau\to +\infty} \tau^{n_1-\frac{n_2}{m_2}}\abs{\zeta(A_2, n_1)- F_{A_2}(\tau,n_1)}  +\\
 C&\limsup_{\tau \to +\infty}
 \tau^{n_1-1-\frac{n_2}{m_2} }  F_{A_2}(\tau,n_1-1). 
\end{align*}
Since $n_1>\frac{n_2}{m_2}>n_1-1$, Proposition \ref{pipi:asin} implies
\[
\zeta(A_2, n_1)- F_1(\tau) =\mathcal{O}\pt{\tau^{\frac{n_2}{m_2} -n_1}},
\quad
F_{A_2}\pt{\tau, n_1-1}=F_3\pt{\tau} = \mathcal{O}\pt{\tau^{\frac{n_2}{m}-n_1+1}}.
\]
Therefore,
\[
	\limsup_{\tau \rightarrow +\infty} 
	\frac{\abs{N_A(\tau)- \frac{C_1}{n_1} \zeta(A_2, n_1) \tau^{n_1}}}{\tau^{\frac{n_2}{m_2}}} <+\infty.
\]
The proof is complete.
\end{proof}
\noindent We can now prove our main result.
\begin{theo}
\label{thm:teoprin}
Let $M_1, M_2$ be two closed manifolds of dimension $n_1$, $n_2$, respectively. 
Let $A= A_1 \otimes A_2$, where $A_j \in \Lcl {m_j}{M_j}$,  $m_j>0$, $j=1, 2$, are
positive, self-adjoint, invertible operators, with $\dfrac{n_1}{m_1} > \dfrac{n_2}{m_2}$. Then,
for $\tau \rightarrow + \infty$,
\begin{equation*}
N_A(\tau) = \left\{
\begin{array}{lcl}
\dfrac{C_1}{n_1} \zeta\left(A_2, \dfrac{n_1}{m _1}\right) \tau^{\frac{n_1}{m_1}} + \mathcal{O}\left(\tau^{\frac{n_1-1}{m_1}}\right)  & \mbox{if} & \dfrac{n_2}{m_2} < \dfrac{n_1-1}{m_1}, \\
\dfrac{C_1}{n_1}  \zeta\left(A_2, \dfrac{n_1}{m _1}\right) \tau^{\frac{n_1}{m_1}} + \mathcal{O}\left(\tau^{\frac{n_1-1}{m_1}}
\log \tau\right)  & \mbox{if} & \dfrac{n_2}{m_2} = \dfrac{n_1-1}{m_1}, \rule{0mm}{7mm}\\
\dfrac{C_1}{n_1}  \zeta\left(A_2, \dfrac{n_1}{m _1}\right) \tau^{\frac{n_1}{m_1}} + \mathcal{O}\left(\tau^{\frac{n_2}{m_2}}\right)  & \mbox{if} & \dfrac{n_2}{m_2} > \dfrac{n_1-1}{m_1}, \rule{0mm}{7mm} 
\end{array} \right. 
\end{equation*}
where $C_1$ is given by \eqref{eq:C1}.
%
%\begin{equation*}
%C_1 = \dfrac{1}{(2\pi)^{n_1}} \int_{M_1}\int_{\mathbb{S}^{n_1-1}} 
%\frac{d\theta_1 dx_1}{[a_{m_1}(x_1, \theta_1)]^{\frac{n_1}{m_1}}} . 
%\end{equation*}
%
\end{theo}
\begin{proof}
Without loss of generality, we can assume $m_1 = 1$, possibly considering an 
appropriate  
power of $A$, see \cite{BA12}. Moreover, again without loss of the generality,
we can assume that all the 
eigenvalues are strictly larger than one, so that the Assumptions \ref{assu1st}
are fulfilled. Then, the claim follows from Lemma \ref{lem:iuno}. 
\end{proof}
The case of the tensor product of two Shubin operators can be treated in a completely similar fashion,
using Proposition \ref{pipi:asinSh} in place of Proposition \ref{pipi:asin},
and the Weyl law \eqref{eq:weylshubin} which holds in this setting.
\begin{theo}
\label{thm:teoprinbis}
Let $P= P_1 \otimes P_2$ and $P_j \in \Gcl {m_j}{\RR^{n_j}}$,  $m_j>0$, $j=1,2$, be
positive, self-adjoint, invertible operators, with $\dfrac{2n_1}{m_1} > \dfrac{2n_2}{m_2}$. Then,
for $\tau \rightarrow + \infty$,
\begin{equation*}
N_P(\tau) = \left\{
\begin{array}{lcl}
\dfrac{K_1}{2n_1} \zeta\left(P_2, \dfrac{2n_1}{m _1}\right) \tau^{\frac{2n_1}{m_1}} + 
\mathcal{O}\left(\tau^{\frac{2n_1-1}{m_1}}\right)  & 
\mbox{if} & \dfrac{2 n_2}{m_2} < \dfrac{2n_1-1}{m_1},\\
\dfrac{K_1}{2n_1}  \zeta\left(P_2, \dfrac{2n_1}{m_1}\right) \tau^{\frac{2n_1}{m_1}} + 
 \mathcal{O}\left(\tau^{\frac{2n_1-1}{m_1}} \log \tau\right) 
& \mbox{if} & \dfrac{2 n_2}{m_2} = \dfrac{2n_1-1}{m_1},  \rule{0mm}{7mm}\\
\dfrac{K_1}{2n_1}  \zeta\left(P_2, \dfrac{2n_1}{m_1}\right) \tau^{\frac{2n_1}{m_1}} + 
\mathcal{O}\left(\tau^{\frac{2n_2}{m_2}}\right)  & \mbox{if} & 
\dfrac{2 n_2}{m_2} > \dfrac{2n_1-1}{m_1},  \rule{0mm}{7mm} 
\end{array} \right.
\end{equation*}
 where 
\begin{equation*}
K_1 = \dfrac{1}{(2\pi)^{2 n_1}} \int_{\mathbb{S}^{2n_1-1}} 
\frac{d\theta_1}{[p_{m_1}(\theta_1)]^{\frac{2 n_1}{m_1}}}. 
\end{equation*}
\end{theo}

%%%%%%%%%%%%%%%%%%%%%%%%%%%%%%%%%%%%%%%%%%%%%%%%%%%%%%%%%%%%%
 \section{Spectral asymptotics for the tensor product of $r$ operators}\label{sec:prodn}
As in the previous sections, to avoid redundancy we will prove in detail our results for tensor products of $r$ factors
only in the case of operators belonging to the H\"ormander calculus on closed manifolds. We will then omit
the proof of the analogous Theorem \ref{thm:teoprinndim} for the case of operators belonging to the Shubin calculus,
which can be obtained by similar arguments.

The main tool in the study of the  extension of Theorem \ref{thm:teoprin} to the product of $r\ge2$ factors is a refined
version of Proposition \ref{pipi:asin}. Let us first state the hypotheses.
\begin{ass}\label{assu2nd}
  \begin{align*}
    &\text{$M_1,\ldots,M_r$ smooth closed manifolds of dimensions $n_1,\dots,n_r$, respectively};
    \\
    &A=A_1\otimes \cdots\otimes A_r, \quad  A_j\in \Lcl{m_j}{M_j}, \, m_j>0, \quad j=1,\ldots,r; \\
    & A_j \text{ positive, self-adjoint, elliptic}, \quad j=1,\ldots,r;
    \\
    &\sigma\pt{A_j}=\ptg{{^j}\mu_{k_j}}_{k_j\in \NN}, 
     \quad {^j}\mu_1>1, \quad     j=1,\ldots,r.
  \end{align*}
\end{ass}

  \begin{pipi}\label{pipi:asinndim}
  Let $A,A_j$, $j=1, \ldots,r$, be as in Assumptions \ref{assu2nd}.
  Set
  \[
   	p= \max\ptg{\frac{n_1}{m_1},\dots,\frac{n_r}{m_r}}, \quad S=\ptg{j\in\{1,\dots,r\} \colon \frac{n_j}{m_j}=p}, 
	\quad s=\sharp S,
  \]
  and define, for $\tau\to+\infty$,
  \[
  	F_A(\tau,c)=\sum_{{^1} \mu_{k_1} \,\cdot \, \ldots \, \cdot\, {^r} \mu_{k_r}<\tau} \frac{1}{
      ({^1}\mu_{k_1})^c \cdot \ldots \cdot ({^r}\mu_{k_r})^c}
      =\begin{cases}
      	F_1(\tau)&\text{if } p<c,
	\\
	F_2(\tau)&\text{if } p=c,
	\\
	F_3(\tau)&\text{if } p>c.
      \end{cases}
  \]
  Then,  
\begin{align*}
    &\limsup_{\tau\to +\infty} \frac{ \abs{\prod_{j=1}^r \zeta(A_j, c)-F_1(\tau)}}{\tau^{p-c}\, \pt{\log \tau}^{s-1}}= \kappa_1, \\
    &\limsup_{\tau\to +\infty} \frac{F_2\pt{\tau}}{\pt{\log\tau}^{s}}=\kappa_2,\\
    &\limsup_{\tau\to +\infty} \frac{F_3\pt{\tau}}{\tau^{p-c}\pt{\log \tau}^{s-1}}=\kappa_3,    
\end{align*}
 that is, for $\tau\to+\infty$,
 \[
 	\prod_{j=1}^r \zeta(A_j, c)-F_1(\tau) =\mathcal{O}\pt{\tau^{p-c}\pt{\log \tau}^{s-1}},
 \]
 \[
 	F_2(\tau) =\mathcal{O}\pt{\pt{\log\tau}^{s}},\;
 	F_3(\tau) = \mathcal{O}\pt{\tau^{p-c}\pt{\log\tau}^{s-1} }.
 \]
\end{pipi}
\begin{proof} We will make use of the straightforward inequality
\begin{equation}
      \label{eq:ineq2nd}
	F_A(\tau,c) = 
	\sum_{{^1} \mu_{k_1} \,\cdot \, \ldots \, \cdot\, {^r} \mu_{k_r}<\tau} 
	\frac{1}{({^1}\mu_{k_1})^c \cdot \ldots \cdot ({^r}\mu_{k_r})^c}
	\leq \prod_{j=1}^r \,\sum_{{^j} \mu_{k_j} <\tau} \frac{1}{({^j}\mu_{k_j})^c},
\end{equation}
as well as of the following consequence of the absolute convergence of the involved series,
\begin{equation}
      \label{eq:zetaseries}
	 \prod_{j=1}^r \zeta(A_j,c)= \lim_{\tau\to+\infty}\prod_{j=1}^r \,\sum_{{^j} \mu_{k_j} <\tau} \frac{1}{({^j}\mu_{k_j})^c}
	=\sum \frac{1}{({^1}\mu_{k_1})^c \cdot \ldots \cdot ({^r}\mu_{k_r})^c},
\end{equation}
where $c$ belongs to the holomorphic domain of the functions $\zeta(A_j,z)$, $j=1,\dots,r$, and
the last summation in \eqref{eq:zetaseries} is taken on all the $r$-tuples of eigenvalues
$({^1}\mu_{k_1},\dots,{^r}\mu_{k_r})\in\sigma(A_1)\oplus\cdots\oplus\sigma(A_r)$.

 \begin{itemize}
 \item[\underline{Case $p=c$.}] Let us split the last term in \eqref{eq:ineq2nd} as
    \[
	\prod_{j=1}^r \,\sum_{{^j} \mu_{k_j} <\tau} \frac{1}{({^j}\mu_{k_j})^c}
      = \pt{\prod_{j\not \in S}\, \sum_{{^j} \mu_{k_j} <\tau} \frac{1}{({^j}\mu_{k_j})^c}} \cdot 
      \pt{\prod_{t \in S}\, \sum_{{^t} \mu_{k_t} <\tau} \frac{1}{({^t}\mu_{k_t})^c}}.
    \]
    Recalling that $\frac{n_j}{m_j}<c$ for all $j \notin S$, that is, $c$ belongs to the holomorphic domain
    of $\zeta(A_j, \cdot)$ for $j \notin S$, and that $\frac{n_t}{m_t}=c$ for all $t\in S$, 
    using Proposition \ref{pipi:asin} we have
    \begin{align*}
      \pt{\prod_{j\notin S}\, \sum_{{^j} \mu_{k_j} <\tau} \frac{1}{({^j}\mu_{k_j})^c} } &\cdot 
      \pt{\prod_{t \in S}\, \sum_{^t \mu_{k_t} <\tau} \frac{1}{({^t}\mu_{k_t})^c}}
      \\
      &\hspace{-1cm}=\pt{\prod_{j \not \in S} 
      \zeta(A_j, c)}\cdot \mathcal{O}\pt{\pt{\log \tau}^s}= \mathcal{O}\pt{ \pt{\log \tau}^{s}},
    \end{align*}
    which implies our claim in this case, in view of \eqref{eq:ineq2nd}.\\
       \item[\underline{Case $p>c$.}]
    To simplify notation, we can suppose, without loss of generality, $p=\frac{n_1}{m_1}$.
    Recalling the assumption ${^j} \mu_{k_j}>1$, $j=1,\ldots,r$, we observe that
    \[
      {^1} \mu_{k_1}\cdot \ldots \cdot {^r}\mu_{k_r}<\tau\Leftrightarrow 
     \left[
     \prod_{j=2}^r ({^j} \mu_{k_j})<\tau
     \;\wedge\; 
     1<{^1}\mu_{k_1}<\frac{\tau}{\prod_{j=2}^r ({^j} \mu_{k_j})}
     \right].
    \]
    In fact, the $\Leftarrow$ implication is immediate, while
    \begin{align*}
    	{^1} \mu_{k_1}\cdot \ldots \cdot {^r}\mu_{k_r}<\tau&\;\wedge\;{^1} \mu_{k_1}>1
	\\
	&\Rightarrow
	1<{^1} \mu_{k_1}<\frac{\tau}{\prod_{j=2}^r ({^j} \mu_{k_j})}
	\\
	&\Rightarrow 
	1<{^1} \mu_{k_1}<\frac{\tau}{\prod_{j=2}^r ({^j} \mu_{k_j})}
	\;\wedge\;
	\prod_{j=2}^r ({^j} \mu_{k_j})<\tau.
    \end{align*}
    Then, we can write
     \begin{align*}
       F_A(\tau,c)&=F_3(\tau)
       \\
       &=\sum_{{^2} \mu_{k_2}\cdot \ldots \cdot {^r}\mu_{k_r}<\tau}
       \frac{1}{({^2} \mu_{k_2})^c \cdot \ldots \cdot ({^r} \mu_{k_r})^c}
       \sum_{1<{^1}\mu_{k_1}<\frac{\tau}{{^2} \mu_{k_2}\cdot \ldots \cdot {^r}\mu_{k_r}}}
       \frac{1}{({^1} \mu_{k_1})^c}
       \\
       &=\sum_{\prod_{j =2}^{r}  ({^j} \mu_{k_j})<\tau} 
       \frac{1}{\prod_{j=2}^r  ({^j} \mu_{k_j})^c }      
       \int_1^{\frac{\tau}{\prod_{j=2}^{r}  ({^j} \mu_{k_j}) }} \frac{1}{\mu^c} 
       dN_{A_1}(\mu).
     \end{align*}
    Switching to $B=(A_1)^\frac{1}{m_1}$, recalling that then $\sigma(B)=\{({^1}\mu_{k_1})^\frac{1}{m_1}\}$ 
    and\footnote{That is,
    \begin{align*}
	 \sum_{1<{^1}\mu_{k_1}<\frac{\tau}{{^2} \mu_{k_2}\cdot \ldots \cdot {^r}\mu_{k_r}}}
       =\frac{1}{({^1} \mu_{k_1})^c}
       &=
 	 \sum_{1<({^1}\mu_{k_1})^\frac{1}{m_1}<
	 \left[\frac{\tau}{{^2} \mu_{k_2}\cdot \ldots \cdot {^r}\mu_{k_r}}\right]^\frac{1}{m_1}}
       =\frac{1}{\left[({^1} \mu_{k_1})^\frac{1}{m_1}\right]^{c\,m_1}}      
     \end{align*}
     similarly to the proof of Proposition \ref{pipi:asin}.} 
    \begin{align*}
    	\int_1^{\frac{\tau}{\prod_{j=2}^{r}  ({^j} \mu_{k_j}) }} \frac{1}{\mu^c} 
       dN_{A_1}(\mu)
       &=
       \int_1^{\left[\frac{\tau}{\prod_{j=2}^{r}  ({^j} \mu_{k_j}) }\right]^\frac{1}{m_1}} \frac{1}{\mu^{c\,m_1}} 
       dN_{B}(\mu),
      \end{align*}
    using \eqref{eq:grigis} with $n_1$ in place of $n$, it turns out that, for $\tau\to+\infty$,
    \begin{align*}
      \nonumber
      F_3(\tau)&=\sum_{\prod_{j=2}^r   ({^j} \mu_{k_j})<\tau} 
        \frac{1}{\prod_{j=2}^r  ({^j} \mu_{k_j})^c}      
        \;\mathcal{O}
        \!\pt{ \pt{\frac{\tau}{ \prod_{j=2}^r ({^j} \mu_{k_j})}      }^{p-c}}
        \\
      &=\left[\sum_{{^2} \mu_{k_2}\cdot\ldots\cdot{^r}\mu_{k_r}<\tau } 
        \frac{1}{({^2} \mu_{k_2})^p\cdot\ldots\cdot({^r}\mu_{k_r})^p}\right] 
        \mathcal{O}\pt{\tau^{p-c} }.
    \end{align*}
     Using the result of the case $p=c$ above, with $s-1$ in place of $s$, we conclude
    \[
     F_3(\tau)= \mathcal{O}\pt{\tau^{p-c}\pt{\log \tau}^{s-1}},
    \]
        as claimed.\\
  \item[\underline{Case $p<c$.}]
    Since $c> \frac{n_j}{m_j}$ for all $j=1, \ldots, r$,
     $c$ belongs to the holomorphic domain of all the functions $\zeta(A_j,z)$, $j=1,\dots,r$.
 Then, by \eqref{eq:zetaseries}, in this case we have, for all $\tau$,
    \begin{align*}
      \prod_{j=1}^r &\zeta\pt{A_j, c}-F_A(\tau,c)=\prod_{j=1}^r \zeta\pt{A_j, c}-F_1(\tau)
      \\
      &=\sum \frac{1}{({^1}\mu_{k_1})^c \cdot \ldots \cdot ({^r}\mu_{k_r})^c}
      - 
      \sum_{{^1} \mu_{k_1} \,\cdot \, \ldots \, \cdot\, {^r} \mu_{k_r}<\tau} 
	\frac{1}{({^1}\mu_{k_1})^c \cdot \ldots \cdot ({^r}\mu_{k_r})^c}
	\\
	&=
	    \sum_{{^1} \mu_{k_1} \,\cdot \, \ldots \, \cdot\, {^r} \mu_{k_r}\ge\tau}
	    \frac{1}{({^1}\mu_{k_1})^c \cdot \ldots \cdot ({^r}\mu_{k_r})^c}.
     \end{align*}
    We will prove the claim by induction on the number of operators. 
    The case $r=2$ is proven in Proposition \ref{pipi:asin}. Let us then suppose that the desired
    estimate holds true for a tensor product of $r-1$ operators, $r\ge 2$,
    and let us prove that it holds true also for a tensor product of $r$ operators.
    
    We can again suppose, without loss of generality, $p=\frac{n_1}{m_1}$.
    Since, clearly,
    \begin{align*}
    	{^1} \mu_{k_1} &\,\cdot \, \ldots \, \cdot\, {^r} \mu_{k_r}\ge\tau
	\Rightarrow
	\\
	&{^1} \mu_{k_1} \ge\frac{\tau}{  \prod_{j=2}^r  ({^j} \mu_{k_j}) }
	\wedge
	\left[ \prod_{j=2}^r  ({^j} \mu_{k_j}) < \tau \;\vee\; \prod_{j=2}^r  ({^j} \mu_{k_j}) \ge\tau\right],
   \end{align*}
    we can write
    \begin{align}
     \nonumber
     \prod_{j=1}^r&\zeta(A_j,c)-F_1\pt{\tau}=
     \\ 
     \label{eq:part1}
     &\phantom{+}\sum_{\prod_{j=2}^r  ({^j} \mu_{k_j})<\tau}
                  \frac{1}{\prod_{j=2}^r  ({^j} \mu_{k_j})^c}
                 \int_{\frac{\tau}{\prod_{j=2}^r  ({^j} \mu_{k_j}) }}^{+\infty}
                 \frac{1}{\mu^c}dN_{A_1}\pt{\mu} \\
                  \label{eq:part2}
                  &+ \sum_{\prod_{j=2}^r  ({^j} \mu_{k_j})\geq \tau}
                  \frac{1}{\prod_{j=2}^r  ({^j} \mu_{k_j})^c}
                 \int_{\frac{\tau}{\prod_{j=2}^r  ({^j} \mu_{k_j}) }}^{+\infty}
                 \frac{1}{\mu^c}dN_{A_1}\pt{\mu}.                  
    \end{align}
    Let us first consider \eqref{eq:part1}. Arguing as in the previous case $p>c$, and using the case $p=c$
    with $s-1$ in place of $s$, we find, for $\tau\to+\infty$,
      \begin{align*}
                  \sum_{\prod_{j=2}^r  ({^j} \mu_{k_j})<\tau}
                  &\frac{1}{\prod_{j=2}^r  ({^j} \mu_{k_j})^c}
                 \int_{\frac{\tau}{\prod_{j=2}^r  ({^j} \mu_{k_j}) }}^{+\infty}
                 \frac{1}{\mu^c}dN_{A_1}\pt{\mu}
                  \\
                  =\; &\sum_{\prod_{j=2}^r  ({^j} \mu_{k_j})<\tau}
                  \frac{1}{\prod_{j=2}^r  ({^j} \mu_{k_j})^c}
                 \;\mathcal{O}\!\pt{\pt{\frac{\tau}{\prod_{j=2}^{r} \phantom{a}^j \mu_{k_j}} }^{p-c} }\\
                 = \; & \mathcal{O}\pt{\tau^{p-c}\pt{\log\tau}^{s-1}},
    \end{align*}
    which is the desired estimate. We now show that \eqref{eq:part2} fulfills the same estimate. 
    Using the fact that $\zeta(A_1,c)$ is finite, we can estimate \eqref{eq:part2} as
    \begin{align}
	\nonumber
	&\sum_{\prod_{j=2}^r  ({^j} \mu_{k_j})\ge\tau}
                  \frac{1}{\prod_{j=2}^r  ({^j} \mu_{k_j})^c}
                 \int_{\frac{\tau}{\prod_{j=2}^r  ({^j} \mu_{k_j}) }}^{+\infty}
                 \frac{1}{\mu^c}dN_{A_1}\pt{\mu}\\
        \label{eq:eqindu}
        \le& \sum_{\prod_{j=2}^r  ({^j} \mu_{k_j})\ge\tau}
                  \frac{1}{\prod_{j=2}^r  ({^j} \mu_{k_j})^c}
                  \zeta(A_1,c).          
    \end{align}
    By the inductive hypothesis, we see that \eqref{eq:eqindu} is
    $\mathcal{O}\pt{\tau^{\widetilde{p}-c} \pt{\log\tau}^{\widetilde{s}-1}}$ for $\tau\to+\infty$,
    where $\widetilde{p}=\max\ptg{\frac{n_j}{m_j}}_{j=2}^r\leq p$ and
    $\widetilde{s}<s$. Therefore, it is also
    %$ \mathcal{O}\pt{\tau^{\widetilde{p}-c} \pt{\log\tau}^{\widetilde{s}-1}} \subset$
     $\mathcal{O} \pt{\tau^{p-c} \pt{\log\tau}^{s-1}}$ for $\tau\to+\infty$,
     and the same of course holds for \eqref{eq:part2}, in view of the above estimate.
 \end{itemize}
The proof is complete.
\end{proof}
\begin{ass}\label{assu3rd}
Let $A, A_1,\ldots,A_r$ be as in Assumptions \ref{assu2nd}, and suppose that
there exists $l\in \ptg{1, \ldots, r}$ such that 
\[
	\frac{n_l}{m_l}>\max\ptg{\frac{n_j}{m_j}}_{j\in \{1, \ldots, r\}\setminus \{l\}}.
\]
\end{ass}
For notational simplicity, in the next two statements we also assume, without loss of generality, that $l=1$.
As in the previous section, we first consider the case when $m_1=1$.
We will denote by $\mu_\jl$ the product $\prod_{j=2}^r {^j}\mu_{k_j}$,
where $\jl$ denotes the multi-index $(k_2, \ldots,k_r )\in\NN^{r-1}$.
The following proposition is an extension of Proposition \ref{pipi:iunounon}.
\begin{pipi}\label{pipi:iunounon}
Let $A,A_1, \ldots, A_r$ be as in Assumptions \ref{assu3rd}. Then,
\begin{align*}
  N_{A} (\tau) = \sum_{\mu_{\jl}< \tau } 
  \pt{ \frac{C_1}{n_1} \pt{ \frac{\tau}{\mu_{\jl}}}^{n_1} + 
   \frac{1} { \mu_{\jl}^ {n_1-1} }  r_{\jl} \pt{\tau}},
\end{align*}
where $C_1$ is given by \eqref{eq:C1} and $r_{\jl}$ is $\mathcal{O}\pt{\tau^{n_1-1}}$,
 uniformly  with respect to $\mu_{\jl}$, for any $\jl$. That is, there exists a positive constant $C$ such that
 \[
  |r_{\jl}\pt{\tau}|\le C \tau^{n_1-1}, \quad \text{for all } \jl \in \NN^{r-1}.
 \]
\end{pipi}
Proposition \ref{pipi:iunounon} implies the next lemma, which is a multidimensional version of Lemma
\ref{lem:iuno}. We omit the proof, since the argument is analogue to the one used to prove Lemma 
\ref{lem:iuno}, similarly to what has been done in the proof of Proposition \ref{pipi:asinndim}.
\begin{lem}\label{lem:iunon}
Let $A,A_1, \ldots, A_r$ be as in Assumptions \ref{assu3rd}. Let us suppose that $m_1=1$
and $n_1> \frac{n_j}{m_j}$, $j=2, \ldots, r$, and set
\[
 p=\max\ptg{\frac{n_j}{m_j}}_{j=2,\ldots,r}, \quad S=\ptg{j=2,\ldots, r \colon \frac{n_j}{m_j}=p}, \quad s= \sharp S.
 \]
 Then we have, for $\tau \rightarrow +\infty$,
\begin{equation*}
N_{A}(\tau) = \left\{
\begin{array}{lcl}\displaystyle
C_A\,
\tau^{n_1} + \mathcal{O}\left(\tau^{n_1-1}\right) 
& \mbox{if} & p< n_1-1, \\
\displaystyle
C_A\,
                     \tau^{n_1} + \mathcal{O}\pt{\tau^{n_1-1} \pt{\log\tau}^{s} 
                     }
& \mbox{if} & p=n_1-1, \rule{0mm}{7mm} \\
\displaystyle
C_A\,
                     \tau^{n_1} + \mathcal{O}\pt{\tau^{ p} \pt{\log\tau}^{s-1}}
& \mbox{if}  & p > n_1-1, \rule{0mm}{7mm}
\end{array} \right. 
\end{equation*}
where 
\[
	C_A=\dfrac{C_1}{n_1} \prod_{j=2}^r \zeta\left(A_j, n_1\right)
\]
and $C_1$ is given by \eqref{eq:C1}.
\end{lem}
Finally, using powers of the operator $A$, 
it is possible to extend the result to the case where all the factors have arbitrary positive order,
which is, together with Theorem \ref{thm:teoprinndimSh} below for the tensor product of
$r$ factors in the Shubin calculus, our next main result.
\begin{theo}\label{thm:teoprinndim}
Let $M_1, \ldots, M_r$ be closed manifolds of dimension $n_1, \ldots, n_r$, respectively.
Let $A= A_1 \otimes \cdots\otimes A_r$, where $A_j \in \Lcl {m_j}{M_j}$,  $m_j>0$, $j=1, \ldots, r$, are
positive, self-adjoint, invertible operators, and assume that there exists 
$l\in \ptg{1, \ldots, r}$ such that $\frac{n_l}{m_l}> \max\ptg{\frac{n_j}{m_j}}_{j\in 
\ptg{1, \ldots, r}\setminus \ptg{l}}$. Set
\[
 p =\max\ptg{ \frac{n_j}{m_j}}_{j\in \{1, \ldots, r \}\setminus \{l\}},\;
 S=\ptg{j=1,\ldots,r, j\not=l\colon \frac{n_j}{m_j}=p}, \; s=\sharp S.
\]
Then, for $\tau \rightarrow + \infty$, 
\begin{equation*}
N_{A}(\tau) = \left\{
\begin{array}{lcl}\displaystyle
C_A\,
\tau^{\frac{n_l}{m_l}} + \mathcal{O}\left(\tau^{\frac{n_l-1}{m_l}}\right)
& \mbox{if} & p<\dfrac{n_l-1}{m_l}, \\
\displaystyle
C_A\,
                     \tau^{\frac{n_l}{m_l}} + \mathcal{O}\pt{\tau^\frac{n_l-1}{m_l} \,\pt{\log\tau}^s  }
                     & \mbox{if} & p=\dfrac{n_l-1}{m_l}, \rule{0mm}{7mm} \\
\displaystyle
C_A\,
                     \tau^{\frac{n_l}{m_l}} + \mathcal{O}\pt{\tau^{
                     p }\pt{\log\tau}^{s-1}  }
                     & \mbox{if}  & p > \dfrac{n_l-1}{m_l}, \rule{0mm}{7mm}
\end{array} \right. 
\end{equation*}
where 
\[
 C_A=\dfrac{C_l}{n_l} \prod_{\substack{j=1,\ldots,r \\ j\not= l}} \zeta\left(A_j, \frac{n_l}{m _l}\right),
 \quad
 C_l= \frac{1}{\pt{2\pi}^{n_l}} \int_{M_l}\int_{\mathbb{S}^{n_l-1}}
 \frac{d\theta_l dx_l}{[a_{m_l}\pt{x_l, \theta_l}]^{\frac{n_l}{m_l}}} . 
\]
\end{theo}

\begin{theo}\label{thm:teoprinndimSh}
Let $P=P_1\otimes \cdots\otimes P_r$ and $P_j\in \Gcl{m_j}{\RR^{n_j}}$, $m_j>0$,
$j=1,\ldots,r$, be positive, self-adjoint, invertible operators, and assume that 
there exists $l\in \ptg{1, \ldots, r}$ such that $\frac{2n_l}{m_l}> 
\max\ptg{\frac{2n_j}{m_j}}_{j\in 
\{1, \ldots, r\}\setminus \{l\}}$. Set
\[
 p =\max\ptg{ \frac{2n_j}{m_j}}_{j\in \{1, \ldots, r \}\setminus \{l\}},\;
 S=\ptg{j=1,\ldots,r, j\not=l\colon \frac{2n_j}{m_j}=p}, \; s=\sharp S.
\]
Then, for $\tau \rightarrow + \infty$, 
\begin{equation*}
N_{P}(\tau) = \left\{
\begin{array}{lcl}\displaystyle
K_P\,
\tau^{\frac{2n_l}{m_l}} + \mathcal{O}\left(\tau^{\frac{2n_l-1}{m_l}}\right)
& \mbox{if} & p<\dfrac{2n_l-1}{m_l}, \\
\displaystyle
K_P\,
                     \tau^{\frac{2n_l}{m_l}} + \mathcal{O}\pt{\tau^\frac{2n_l-1}{m_l} \,\pt{\log\tau}^s  }
                     & \mbox{if} & p=\dfrac{2n_l-1}{m_l}, \rule{0mm}{7mm} \\
\displaystyle
K_P\,
                     \tau^{\frac{2n_l}{m_l}} + \mathcal{O}\pt{\tau^{
                     p }\pt{\log\tau}^{s-1}  }
                     & \mbox{if}  & p > \dfrac{2n_l-1}{m_l}, \rule{0mm}{7mm}
\end{array} \right. 
\end{equation*}
where 
\[
K_P= \dfrac{V_l}{2n_l} \prod_{\substack{j=1,\ldots,r \\ j\not= l}} \zeta\left(P_j, \frac{2n_l}{m _l}\right),\quad
V_l=\frac{1}{\pt{2 \pi}^{2n_l}} \int_{\mathbb{S}^{2n_l-1}}\frac{d\theta_l}{[p_{m_l}(\theta_l)]^{\frac{2n_l}{m_l}} }.
\]
\end{theo}

%%%%%%%%%%%%%%%%%%%%%%%%%%%%%%%%%%%%%%%%%%%%
\section{Sharpness of the result}\label{sec:contro}
  In this section we show that the estimates obtained in Theorem \ref{thm:teoprin} are sharp.
  To begin, we choose two pseudodifferential operators on spheres,  whose spectrum we can describe
  explicitly. Namely, we set
  \[
    A_1= \pt{-\Delta_{\mathbb{S}^2}+2} - 2 \pt{ -\Delta_{\mathbb{S}^2}+\frac{1}{4}}^{\frac{1}{2}}\in\Lcl{2} {\mathbb{S}^2},
    \;
    A_2= - \Delta_{\mathbb{S}^1}+1\in\Lcl{2}{\mathbb{S}^1},
  \]
  where $A_1$ is considered as an unbounded operator on $L^2(\mathbb{S}^2)$, where
  $\mathbb{S}^2$ is the $2$-dimensional sphere, and
  $A_2$ is considered as an unbounded operator on $L^2(\mathbb{S}^1)$, where
  $\mathbb{S}^1$ is the $1$-dimensional sphere.
  It is well known, see, e.g., \cite[\S 3]{SH87}, that
  \begin{align*}
    \sigma\pt{-\Delta_{\mathbb{S}^2}}&=\ptg{k^2+k \mid k\in \NN, \; \molt{k^2+k}=\pt{2k+1}},\\
    \sigma\pt{-\Delta_{\mathbb{S}^1}}&=\ptg{n^2\mid n\in \NN, \; \molt{n^2}=2},
  \end{align*}
  where $\molt{\tau}$ is the multiplicity of the eigenvalue $\tau$. 
  Therefore, by the functional calculus of operators,
  \begin{align}
   \label{eq:specA1}
    \sigma\pt{A_1}&=\ptg{k^2-k+1 \mid k\in \NN, \; \molt{k^2-k+1}=\pt{2k+1}},\\
    \label{eq:specA2}
    \sigma\pt{A_2}&=\ptg{n^2+1\mid n\in \NN, \; \molt{n^2+1}=2},
  \end{align}
  since the eigenfunction of $A_1$ and $-\Delta_{\mathbb{S}^2}$ are the same.
  Notice that all the eigenvalues of $A_1$ are larger then $1$, therefore
  \begin{align}
   \label{eq:null1}
   N_{A_1}(\tau)=0, \quad \tau\leq 1.
  \end{align}
  Knowing precisely the eigenvalues of $A_1$ together with their multiplicities, we can write, for $\tau>1$,
  \begin{align*}
    N_{A_1}(\tau)=&\sum_{k^2-k+1<\tau} \molt{k^2-k+1}\\
     =&\sum_{k^2-k+1<\tau} \pt{2k+1}=\sum_{k=0}^{\bar{k}} \pt{2k+1}
   \end{align*}
   where
   \begin{align*}
   \bar{k}^2-\bar{k}+1< \tau \leq \pt{\bar{k}+1}^2-\pt{\bar{k}+1}+1=\bar{k}^2+\bar{k}+1, \tau>1.
   \end{align*} 
  That is,
  \begin{multline}
    \label{eq:equivcount}
   	N_{A_1}(\tau)=\sum_{k=0}^{\bar{k}}\pt{2k+1}=
    	\sum_{k^2+k\leq \bar{k}^2+\bar{k}}\molt{k^2+k}= 
    	N_{-\Delta_{\mathbb{S}^2}}\pt{\bar{k}^2+\bar{k}+\frac{1}{2}},
  \end{multline}
  provided that
  \[
	\bar{k}^2-\bar{k}+1< \tau \leq 
	\pt{\bar{k}+1}^2-\pt{\bar{k}+1}+1=\bar{k}^2+\bar{k}+1, \tau>1.
  \]
  Using a well known result on the counting function of the Laplacian on the spheres (see \cite{SH87}), we
  have, for each $\bar{k} \in \NN$,
  \[
    N_{-\Delta_{\mathbb{S}^2}}\pt{\bar{k}^2+\bar{k}+\frac{1}{2}}= \bar{k}^2+2\bar{k}+1.
  \]
  So, in view of \eqref{eq:equivcount}, supposing $\tau>1$, we find
  \begin{multline*}
    N_{A_1}(\tau)= \bar{k}^2+2\bar{k}+1, \\
    \quad  \bar{k}^2-\bar{k}+1< \tau \leq \pt{\bar{k}+1}^2-\pt{\bar{k}+1}+1=\bar{k}^2+\bar{k}+1.
  \end{multline*}
  The asymptotic expansion \eqref{eq:weylcpt} implies that
  \[
    N_{A_1}\pt{\tau}= \tau+ R\pt{\tau}, \quad R= \mathcal{O}\pt{\tau^{\frac{1}{2}}}.
  \]
  We can then obtain a bound for $R(\tau)$:
  \begin{align*}
    R(\tau)&= N_{A_1}\pt{\tau}- \tau\\
              &= \bar{k}^2+2\bar{k}+1- \tau, \quad  \bar{k}^2-\bar{k}+1< \tau \leq \bar{k}^2+\bar{k}+1. 
  \end{align*}
  Therefore, for $\tau>16$,
  \[
    R(\tau)\geq \bar{k}^2+2\bar{k}+1 -\bar{k}^2-\bar{k}-1=\bar{k}> \frac{3\sqrt{\tau}}{4},
  \]
  which implies, in particular, that the remainder is positive for $\tau>16$. 
  We also have
  \[
   R(\tau)< \bar{k}^2+2\bar{k}+1- \bar{k}^2+\bar{k}-1=2\bar{k}< 4 \sqrt{\tau},
  \]
  and we can conclude that
  \begin{align}
    \label{eq:limbasso}
    \frac{3{\sqrt{\tau}}}{4} \leq R(\tau)\leq 4 \sqrt{\tau}, \quad \tau>16.
  \end{align}
  Summing up, we proved that
  \begin{align}
    \label{eq:countA1}
    &N_{A_1}(\tau)=  \tau + R(\tau),\\
    \label{eq:countA2}
    &N_{A_2}(\tau)=   2 \;\tau^{1/2} + \mathcal{O}(1),
  \end{align}
  where the $R(\tau)$ in \eqref{eq:countA1} satisfies \eqref{eq:limbasso}.
  Notice that both $A_1$ and $A_2$ are elliptic, invertible and positive, so it is possible to consider
  powers of these operators of arbitrary exponent. Now, we  examine separately
  the three different situations that can arise.

  %CASE 
  \subsection*{Case $\frac{n_1}{m_1}> \frac{n_2}{m_2}$ and $\frac{n_1-1}{m_1}>\frac{n_2}{m_2}$}
    Let us consider the operator
    \[
      B= A_1 \otimes A_2^2.
    \]
    Clearly $\frac{n_1}{m_1}=\frac{2}{2}=1> \frac{n_2}{m_2}=\frac{1}{4}$ and 
    $\frac{n_1-1}{m_1}=\frac{1}{2}> \frac{n_2}{m_2}=\frac{1}{4}$, so we are in the first case of
    Theorem \ref{thm:teoprin}, which states that 
    \begin{align}
	\label{eq:contBth}
	N_{B}(\tau)= \zeta(A_2^2,1)\tau+\mathcal{O}\pt{\tau^{1/2}}.
    \end{align}
    By equations \eqref{eq:specA1} and \eqref{eq:specA2} we obtain
    \begin{multline*}
      \sigma\pt{B}=\large\{
      \pt{k^2-k + 1} \pt{n^2+1}^2\mid k, n \in \NN,\\ \molt{ (k^2-k + 1) (n^2+1)^2}=2 (2 k+1) \large \}.
    \end{multline*}
    Therefore, 
    \begin{align}
      \nonumber
      N_{B}(\tau)&= \sum_{\pt{k^2-k + 1} \pt{n^2+1}^2< \tau}^{n\in \NN,\; k\in \NN} \molt{\pt{k^2-k + 1} \pt{n^2+1}^2}\\
		   \nonumber
		   &= \sum_{\pt{k^2-k + 1} \pt{n^2+1}^2< \tau}^{n\in \NN,\; k\in \NN} 2 \pt{2k+1}\\
		   \nonumber
		   &=2 \sum_{\pt{k^2-k+1}<\frac{\tau}{(n^2+1)^2}}^{n\in \NN,\; k\in \NN}
		   \molt{k^2-k+1}\\
		   \label{eq:rid}
		   &= 2 \sum_{(n^2+1)^2<\tau}^{n\in \NN} N_{A_1}\pt{\frac{\tau}{\pt{n^2+1}^2}}\\
		   \label{eq:lambda}
		   &=2 \pt{\sum_{(n^2+1)^2<\tau }^{n\in \NN} \frac{\tau}{(n^2+1)^2}+ 
		   R\pt{\frac{\tau}{(n^2+1)^2}}}.
    \end{align}
    Notice that in \eqref{eq:rid} we have made use of \eqref{eq:null1}
    to reduce the summation.
    Let us now show that the estimate \eqref{eq:contBth} is indeed sharp, that is 
    \[
      \limsup_{\tau\to +\infty} \frac{\abs{ N_{B}(\tau)-  \zeta\pt{A_2^2,1}\tau} }
      { \tau^{1/2}}>0,
    \]
    by direct computation. In view of \eqref{eq:lambda}, we can write
    \begin{align}
	\nonumber
	&\limsup_{\tau\to +\infty} \frac{\abs{ N_{B}(\tau)-  \zeta\pt{A_2^2,1}\tau} }
         { \tau^{1/2}}\\
        \nonumber
        &=\limsup_{\tau\to +\infty} \frac{\abs{ 2 \sum_{(n^2+1)^2<\tau } \pt{\frac{\tau}{(n^2+1)^2}+ 
		   R\pt{\frac{\tau}{(n^2+1)^2}}}-  \zeta\pt{A_2^2,1}\tau} }
         { \tau^{1/2}}\\
        \label{eq:caso1s}
        &=\limsup_{\tau\to +\infty} 
        \frac{\abs{ 2 \sum_{(n^2+1)^2<\tau } \frac{\tau}{(n^2+1)^2}-  \zeta\pt{A_2^2,1}\tau+ 
		   2 \sum_{(n^2+1)^2<\tau }R\pt{\frac{\tau}{(n^2+1)^2}}} }
         { \tau^{1/2}}. 
    \end{align}
    We notice that
    \begin{multline*}
      \limsup_{\tau \to +\infty} 
      \frac{\abs{2\sum_{(n^2+1)^2<\tau } \frac{\tau}{(n^2+1)^2}-  \zeta\pt{A_2^2,1} \tau} }
      { \tau^{1/2}}
      \\
      =\limsup_{\tau\to +\infty}\; \tau^{1/2}(F_{A_2^2}(\tau,1)-\zeta\pt{A_2^2,1} ),
    \end{multline*}
    where we have used the notation introduced in Section \ref{sec:taub}. 
    By Proposition \ref{pipi:asin}, $F_{A_2^2}(\tau,1)-\zeta\pt{A_2^2,1} =\mathcal{O}\pt{\tau^{-\frac{3}{4}}}$,
    therefore\footnote{Actually, here one could prove directly that $F_1(\tau)-\zeta\pt{A_2^2,1}$ 
    is asymptotic to $\tau^{-\frac{3}{4}}$.},
    \[
     \limsup_{\tau \to +\infty} 
      \frac{\abs{2 \sum_{(n^2+1)^2<\tau } \frac{\tau}{(n^2+1)^2}-  \zeta\pt{A_2^2,1} \tau}}
      { \tau^{1/2}}=0.
    \]
    Since, for all $\tau$,
    \[
      \sum_{(n^2+1)^2<\tau } 2\frac{\tau}{(n^2+1)^2}-  \zeta\pt{A_2^2,1} \tau\leq0,
    \]
     \eqref{eq:caso1s} becomes
    \begin{align*}
      \nonumber
       &\limsup_{\tau\to +\infty} \frac{\abs{ N_{B}(\tau)-  \zeta\pt{A_2^2,1}\tau} }
         { \tau^{1/2}}\\
      &\geq - \limsup_{\tau \to +\infty} \frac{{\zeta\pt{A_2^2,1}-2\sum_{(n^2+1)^2<\tau } 
      \frac{\tau}{(n^2+1)^2}}}
      { \tau^{1/2}}+
      2 \limsup_{\tau\to +\infty}  \sum_{(n^2+1)^2<\tau } \frac{\abs{R\pt{\frac{\tau}
      {(n^2+1)^2}}}}{ \tau^{1/2}}\\
        \nonumber
      &\geq  \frac{3}{2}   \limsup_{\tau\to +\infty} \sum_{(n^2+1)^2<\tau }
      \frac{\tau^{1/2}}{ (n^2+1) \tau^{1/2}}\\
      &= \frac{3}{2}\zeta\pt{A_2^2,\frac{1}{2}}.
    \end{align*}
    Here, we have used  the estimates \eqref{eq:limbasso} and that the quantities
    $\frac{n_1}{m_1}=1$ and $\frac{n_1-1}{n_2}=\frac{1}{2}$ are larger than 
    $\frac{n_2}{m_2}=\frac{1}{4}$. The latter implies that $\zeta\pt{A_2^2, \frac{1}{2}}$ is
    a finite, positive quantity\footnote{The convergence
    of the involved series is straightforward.},
    in view of the holomorphic
    properties of the spectral $\zeta$-function of elliptic positive pseudodifferential operators on closed manifolds,
    see \cite{SE67}. This proves the desired result.
		\subsection*{Case $\frac{n_1}{m_1}> \frac{n_2}{m_2}$ and $\frac{n_1-1}{m_1}=\frac{n_2}{m_2}$}
   We consider the operator
   \[
      C= A_1 \otimes A_2.
   \]
   Clearly $\frac{n_1}{m_1}=\frac{2}{2}=1> \frac{n_2}{m_2}=\frac{1}{2}$ and
   $\frac{n_1-1}{m_1}=\frac{1}{2}=\frac{n_2}{m_2}$ so that we are in second 
   case of Theorem \ref{thm:teoprin}, which now states that
   \[
      N_C(\tau)=  \zeta(A_2, 1) \tau + \,\mathcal{O}\pt{\tau^{1/2} \log \tau}.
   \]
   Using \eqref{eq:specA1} and \eqref{eq:specA2} we obtain explicitly the spectrum of $C$, namely
   \[
      \sigma(C)=\ptg{\pt{k^2-k+1}\pt{n^2+1}\mid \molt{\pt{k^2-k+1}\pt{n^2+1}}= 2(2k+1)}.
   \]
   Therefore, using \eqref{eq:null1}, 
   \begin{align}
   \nonumber
     N_{C}(\tau)&= \sum_{\pt{k^2-k+1}\pt{n^2+1}< \tau}^{n\in \NN, k\in \NN} 2 \pt{2k+1}\\
 		  \nonumber
 		  &=2 \sum_{\pt {k^2-k+1} < \frac{\tau}{ n^2+1}}^{n\in \NN, k\in \NN}  
 		   \molt{k^2+k+1}\\
 		  \nonumber
 		  &=2 \sum_{\pt{n^2+1}<\tau}^{n\in \NN}  N_{A_1} \pt{\frac{\tau}{ n^2+1 } }\\ 		  
 		  \label{eq:OA1b}
 		  &=2 \sum_{n^2+1<\tau}^{n\in \NN} \pt{\frac{\tau}{n^2+1}+ 
 		  R\pt{\frac{\tau}{n^2+1}}}. 		  
    \end{align}
    Let us check directly that
    \begin{align}
      \label{eq:tesiD}
      \limsup_{\tau \to +\infty} \frac{ \abs{N_C(\tau)-  \zeta(A_2, 1)\tau}}{ \tau^{1/2} \log \tau}>0.
    \end{align}
    Using \eqref{eq:OA1b} and \eqref{eq:limbasso} we can write
    \begin{align}
	\nonumber
	&\hspace*{-1cm}
	\limsup_{\tau \to +\infty} \frac{ \abs{N_C(\tau)-  \zeta(A_2, 1)\tau}}{ \tau^{1/2} \log \tau}\\
	\nonumber
	&=\limsup_{\tau \to +\infty} \frac{ \abs{2 \sum_{n^2+1<\tau} 
	\pt{\frac{\tau}{n^2+1}+ 
 		  R\pt{\frac{\tau}{n^2+1}}}- \zeta(A_2,1)\tau} }{\tau^{1/2}\log \tau} \\
 	\nonumber
 	&\geq - \limsup_{\tau \to +\infty} \frac{\tau^{1/2}
 	\pt{\zeta(A_2,1)- 2 \sum_{n^2+1<\tau} \frac{1}{n^2+1}}}{\log \tau}  
 	\\
	\nonumber
	&\phantom{=}+\limsup_{\tau\to +\infty} 
 	\frac{3}{4}\tau^{1/2} \frac{ 2 \sum_{n^2+1<\tau} 
 	\frac{1}{ \pt{n^2+1}^{1/2}}}  {\tau^{1/2}\log \tau }	\\
 	&
 	\label{eq:caso3fin}
 	\geq -\limsup_{\tau \to +\infty} \tau^{\frac{1}{2}}\frac{2 \sum_{n^2+1\geq\tau}
 	\frac{1}{n^2+1}}{\log \tau}+ 
 	\limsup_{\tau \to +\infty}\frac{3}{2} 
 	\frac{\sum_{n^2+1<\tau} \frac{1}{\pt{n^2+1}^{1/2}}}{\log \tau}
    \end{align}
    Finally, using the results of Proposition \ref{pipi:asin} (or directly, by
    integral inequalities), we obtain that
    \begin{align*}
      &\limsup_{\tau \to +\infty}\tau^{\frac{1}{2}} \frac{2 \sum_{n^2+1\geq\tau}
 	\frac{1}{n^2+1}}{\log \tau}
      =\lim_{\tau \to +\infty} \tau^{\frac{1}{2}}\frac{2 \sum_{n^2+1\geq\tau}
 	\frac{1}{n^2+1}}{\log \tau}=0.
    \end{align*}
    Moreover,
    \[
      \limsup_{\tau \to +\infty}\frac{3}{2} 
 	\frac{\sum_{n^2+1<\tau} \frac{1}{n^2+1}}{\log \tau}=\frac{3}{4},
    \]
    so that, by means of \eqref{eq:caso3fin}, the desired result is proven also in this second case.
  \subsection*{Case $\frac{n_1}{m_1}> \frac{n_2}{m_2}$ and $\frac{n_1-1}{m_1}<\frac{n_2}{m_2}$}
    In this situation we consider the operator
    \[
      D= A_1 \otimes A_2^{\frac{3}{4}}.
    \]
    Clearly, $\frac{n_1}{m_1}=\frac{2}{2}=1 > \frac{n_2}{m_2}=\frac{2}{3}$ 
    and $\frac{n_1-1}{m_1}=\frac{1}{2}<  \frac{n_2}{m_2}=\frac{2}{3}$, so we are in the third case of
    Theorem \ref{thm:teoprin}, which implies that
    \begin{align}
      \label{eq:contCth}
      N_D(\tau)=  \zeta\pt{A_2^{\frac{3}{4}},1}\tau+ \mathcal{O}\pt{\tau^{\frac{2}{3}}}.
    \end{align}
    It is immediate to observe that
    \begin{multline}
      \sigma(D)=\large\{
      \pt{k^2+k+1}\pt{n^2+1}^{3/4}\mid\\
      \molt{(k^2+k+1)(n^2+1)^{3/4}}= 2\pt{2k+1}\large\}.
    \end{multline}
    Therefore, using again \eqref{eq:null1}, we obtain
    \begin{align}
    \nonumber
     N_{D}(\tau)&= \sum_{\pt{k^2-k+1}\pt{n^2+1}^{3/4}< \tau}^{n\in \NN, k\in \NN} 2 \pt{2k+1}\\
 		  \nonumber
 		  &=2 \sum_{\pt {k^2-k+1} < \frac{\tau}{ \pt{n^2+1}^{3/4} }}^{n\in \NN, k\in \NN} 
 		   \molt{k^2-k+1}\\
 		  \nonumber
 		  &=2 \sum_{\pt{n^2+1}^{3/4}<\tau }^{n\in \NN} N_{A_1} \pt{\frac{\tau}{ \pt{n^2+1}^{3/4} } }\\
 		  \label{eq:OA1}
 		  &= 2 \sum_{\pt{n^2+1}^{3/4}<\tau}^{n\in \NN}\pt{ \frac{\tau}{\pt{n^2+1}^{3/4}}+ 
 		  R\pt{\frac{\tau}{\pt{n^2+1}^{3/4}}} }. 		  
    \end{align}
    Let us now compute directly
    \[
      \limsup_{\tau\to +\infty}\frac{ \abs{ N_D(\tau)-  \zeta\pt{A_2^{3/4},1}\tau}}{\tau^{2/3}}.
    \]
    By \eqref{eq:OA1}, we find
    \begin{align*}
      &\limsup_{\tau\to +\infty}\frac{ \abs{ N_D(\tau)-  \zeta(A_2^{3/4},1)\tau}}{\tau^{2/3}}\\
      &= \limsup_{\tau\to +\infty} \frac{\abs{2 \sum_{\pt{n^2+1}^{3/4}<\tau}\pt{ \frac{\tau}{\pt{n^2+1}^{3/4}}+ 
 		  R\pt{\frac{\tau}{\pt{n^2+1}^{\frac{3}{4}}}} } -  \zeta\pt{A_2^{3/4},1}\tau } } {\tau^{2/3}}\\
      &= \limsup_{\tau\to +\infty}  {\tau^{-2/3}}\cdot
      \left|2 \sum_{\pt{n^2+1}^{3/4}<\tau}
       \frac{\tau}{\pt{n^2+1}^{3/4}}-  \zeta\pt{A_2^{3/4},1}\tau+ \right.
      \\
      &\phantom{\limsup_{\tau\to +\infty}  {\tau^{-2/3}}\cdot|}
 		  \,\left.+2 \sum_{\pt{n^2+1}^{3/4}<\tau} R\pt{\frac{\tau}{\pt{n^2+1}^{\frac{3}{4}}}}
		  \right|.
    \end{align*}
    We also notice that
    \begin{align*}
      &\lim_{\tau \to +\infty} \frac{\abs{2 \sum_{\pt{n^2+1}^{3/4}<\tau}
       \frac{\tau}{\pt{n^2+1}^{3/4}}-  \zeta\pt{A_2^{3/4},1}\tau}}{\tau^{2/3}}\\
      &\hspace*{1cm}=\lim_{\tau \to +\infty} \frac{ \zeta\pt{A_2^{3/4},1}\tau -2 \sum_{\pt{n^2+1}^{3/4}<\tau}
       \frac{\tau}{\pt{n^2+1}^{3/4}} }{\tau^{2/3}}\\ 
      &\hspace*{1cm}=\lim_{\tau \to +\infty}\;2\; \tau^{1/3}   \sum_{\pt{n^2+1}^{3/4}\geq\tau}
       \frac{1}{\pt{n^2+1}^{3/4} }, 
    \end{align*}
    and that
    \[
         \sum_{\pt{n+1}^{3/2}\geq \tau} \frac{1}{\pt{n+1}^{3/2}} \leq 
         \sum_{\pt{n^2+1}^{3/4}\geq\tau}\frac{1}{\pt{n^2+1}^{3/4}}\leq
         \sum_{n^{3/2}\geq \tau} \frac{1}{n^{3/2}}.
    \]
    Using the standard integral criteria of series convergence, one can easily check that
    \[
      \lim_{\tau \to +\infty } \tau^{1/3}\sum_{\pt{n+1}^{3/2}\geq \tau} \frac{1}{\pt{n+1}^{3/2}} =
      \lim_{\tau \to +\infty }\tau^{1/3}\sum_{n^{3/2}\geq \tau} \frac{1}{n^{3/2}} =2.
    \]
    Hence
    \begin{align}
      \label{eq:prinsomma}
      \lim_{\tau \to +\infty}\;2\; \tau^{1/3}   \sum_{\pt{n^2+1}^{3/4}\geq\tau}
       \frac{1}{\pt{n^2+1}^{3/4} } =4.
    \end{align}
    By a similar argument, we also have that
    \begin{align}
      \label{eq:sommarest}
     \lim_{\tau\to +\infty } \tau^{-1/6}\sum_{\pt{n^2+1}^{3/4}<\tau}\frac{1}{\pt{n^2+1}^{3/8}}= 4.
    \end{align}
    In view of \eqref{eq:limbasso}, \eqref{eq:prinsomma} and \eqref{eq:sommarest} 
    we finally obtain
    \begin{align}
       \nonumber
       &\hspace*{-1cm}
       \limsup_{\tau\to +\infty}\frac{ \abs{ N_D(\tau)-  \zeta(A_2^{3/4},1)\tau}}{\tau^{2/3}}\\
       \nonumber
       &\geq
       \limsup_{\tau\to +\infty}\frac{ { N_D(\tau)-  \zeta(A_2^{3/4},1)\tau}}{\tau^{2/3}}\\
       \nonumber
       &= -\lim_{\tau\to +\infty}\;2\; \tau^{1/3}   \sum_{\pt{n^2+1}^{3/4}\geq\tau}
       \frac{1}{\pt{n^2+1}^{3/4} } 
       \\
       \nonumber
       &\phantom{=}+
       \limsup_{\tau\to +\infty} 2
        \frac{ \sum_{\pt{n^2+1}^{3/4}<\tau} R \pt{ \frac{\tau} 
       {\pt{n^2+1}^{3/4} }  }}   { \tau^{2/3} }\\  
       \nonumber
       &\geq -4 +  \frac{3}{2}\limsup_{\tau\to +\infty} \tau^{-1/6}
       \sum_{\pt{n^2+1}^{3/4}<\tau} \frac{1}{\pt{n^2+1}^{3/8}}   \\
       \label{eq:asintC}
       &\geq -4+6=2>0.
    \end{align}
    Equation \eqref{eq:asintC} proves the desired result also in this last case.

\section{Appendix. The Dirichlet divisors problem}\label{sec:app}
  Counting functions of the type \eqref{eq:countA} suggest a spectral approach to
  a prominent type of lattice problem, the so-called Dirichlet divisors problem. 
  Let us suppose that the spectrum of both $A_1$ and $A_2$ in \eqref{eq:countA} is formed
  by all strictly positive natural numbers, each with multiplicity one. Then,
  \[
    N_{A}(\tau)= \sum_{n\cdot m< \tau} 1 = D(\tau).
  \]
  The function $D(\tau)$ is called Dirichlet divisor summatory function and it 
  is straightforward to check that it amounts the number of points with integer 
  coordinates belonging to the first quadrant of the Cartesian plane which lie below the hyperbola $x y= \tau$.
  In 1849, Dirichlet proved that
  \begin{equation}
    \label{eq:dirc}
    D(\tau)= \tau \log  \tau + (2 \gamma -1)\tau+ \mathcal{O}(\tau^{1/2}),
  \end{equation}
  where $\gamma$ is the Euler-Mascheroni constant, namely
  \[
    \gamma= \lim_{\tau\to +\infty} \pt{\sum_{0<n<\tau}\frac{1}{n} -\int_0^\tau \frac{1}{x}dx},
  \]
  or, equivalently,
  \[
    \gamma=\lim_{z\to 1}\pt{z-1}\zeta_R(z),
  \]
  where $\zeta_R(z)$ is the Riemann $\zeta$-function.
  Several papers aimed at finding the sharp remainder term in \eqref{eq:dirc}, see \cite{IKKN06}
  for an overview on this type of problems. Hardy, in \cite{HA16},
  proved that $\mathcal{O}(\tau^{\frac{1}{4}})$ is a lower bound for the remainder in \eqref{eq:dirc}.
  It is conjectured that the sharp estimate  in this case is $\mathcal{O}(\tau^{\frac{1}{4}+\epsilon})$
  or, more precisely, $\mathcal{O}\pt{\tau^{1/4}\log \tau }$.
  The best known result, due  to Huxley, is that the remainder is
  $\mathcal{O}(\tau^\alpha \pt{\log \tau}^{\beta+1})$, where
    \begin{align*}
   \alpha= \frac{131}{416}\sim 0,3149\ldots \quad \beta=\frac{18627}{8320}\sim 2,2513\ldots.
  \end{align*}
  In order to have a spectral interpretation of the Dirichlet divisor problem, a global bisingular calculus based on Shubin calculus has been introduced in \cite{BGPR13}.
  Then, the following Hermite-type operator  
  \[
   H_j= \frac{1}{2}\pt{-\partial_{x_j}^2+ x_j^2} +\frac{1}{2}, \quad j=1,2,
  \]
  has been examined.
  Using Hermite polynomials, it turns out that $\sigma(H_j)=\{n\}_{n\in \NN^*}$, $j=1,2$, 
  and each eigenvalue has 
  multiplicity one. Therefore
  $\sigma(H_1 \otimes H_2)=\ptg{n\cdot m}_{(n,m)\in (\NN^*)^2}$ and
  \[
    N_{H_1 \otimes H_2}(\tau)= D(\tau).
  \]
  This clear spectral meaning of the Dirichlet divisor problem was one of the main motivation
  of the papers \cite{BGPR13,GPRV14}. For the connection between Dirichlet divisor problem and 
  standard bisingular operators on  the product of closed manifolds see \cite{BA12}.
  Actually, since we deal with the non-symmetric case,
  it is not possible to attack directly the traditional Dirichlet divisor problem through the approach described in the
  previous sections, while our techniques
  are well suited to treat generalized anisotropic Dirichlet divisors problems like, for instance,
  \[
    N_{H_1^{\alpha}\otimes H_2^{\beta}} (\tau)=\sum_{ n^{\alpha} \cdot m^{\beta}<\tau}1, \quad \alpha\neq \beta.
  \]
  In \cite{GPRV14} it is proven that
  \begin{equation}
    \label{eq:controrod}
    N_{H_1^{\alpha}\otimes H_2^{\beta}}\pt{\tau}= 
    \zeta\pt{\frac{\alpha}{\beta}}\tau^{\frac{1}{\beta}}+\zeta\pt{\frac{\beta}{\alpha}}\tau^{\frac{1}{\alpha}}+
    \mathcal{O}\pt{\tau^{\frac{1}{\alpha+\beta}}},
  \end{equation}
  where $\zeta$ is the meromorphic continuation of the Riemann $\zeta$-function.
  Notice that \eqref{eq:controrod} proves the sharpness of the result stated in Theorem \ref{thm:teoprinbis} in the case $\frac{2n_2}{m_2}>\frac{2n_1-1}{m_1}$.
  
%  

%\bibliography{Bibliografia}

\begin{thebibliography}{BGRP13}

\bibitem[ANPS09]{ANS09}
Wolfgang Arendt, Robin Nittka, Wolfgang Peter, and Frank Steiner, \emph{Weyl
  law: Spectral properties of the laplacian in mathematics and physics}, In: Mathematical Analysis of Evolution, Information, and Complexity, Wiley-VCH Verlag GmbH Co. KGaA, Weinheim, (2009).

\bibitem[Ara88]{AR88}
Junichi Aramaki, \emph{On an extension of the {I}kehara {T}auberian theorem},
  Pacific J. Math. \textbf{133} (1988), no.~1, 13--30.

\bibitem[AS68]{AS68}
M.~F. Atiyah and I.~M. Singer, \emph{The index of elliptic operators. {I}},
  Ann. Math. (2) \textbf{87} (1968), 484--530. 

\bibitem[Bat12]{BA12}
Ubertino Battisti, \emph{Weyl asymptotics of bisingular operators and
  {D}irichlet divisor problem}, Math. Z. \textbf{272} (2012), no.~3-4,
  1365--1381. 

\bibitem[BC11]{BC11}
Ubertino Battisti and Sandro Coriasco, \emph{Wodzicki residue for operators on
  manifolds with cylindrical ends}, Ann. Global Anal. Geom. \textbf{40} (2011),
  no.~2, 223--249.

\bibitem[BGRP13]{BGPR13}
U.~Battisti, T.~Gramchev, L.~Rodino, and S.~Pilipovi{\'c}, \emph{Globally
  bisingular elliptic operators}. In: Operator Theory, Pseudo-differential
  equations, and Mathematical Physics, Operator Theory: Advanced and Applications, vol. 228,
  Birkh\"auser/Springer Basel AG, Basel, 2013, pp.~21--38. 

\bibitem[BN03]{BN03}
Paolo Boggiatto and Fabio Nicola, \emph{Non-commutative residues for
  anisotropic pseudo-differential operators in {$\mathbb R\sp n$}}, J. Funct.
  Anal. \textbf{203} (2003), no.~2, 305--320. 

\bibitem[CM13]{CM13}
Sandro Coriasco and Lidia Maniccia, \emph{On the spectral asymptotics of
  operators on manifolds with ends}, Abstr. Appl. Anal., vol. 2013, (2013), ID 909782,
  21. 

\bibitem[DD13]{DD13}
Kiril Datchev and Semyon Dyatlov, \emph{Fractal {W}eyl laws for asymptotically
  hyperbolic manifolds}, Geom. Funct. Anal. \textbf{23} (2013), no.~4,
  1145--1206.

\bibitem[GL02]{LO02}
Juan~B. Gil and Paul~A. Loya, \emph{On the noncommutative residue and the heat
  trace expansion on conic manifolds}, Manuscripta Math. \textbf{109} (2002),
  no.~3, 309--327. 

\bibitem[GPRVar]{GPRV14}
Todor Gramchev, Stevan Pilipovi\'c, Luigi Rodino, and Jasson Vindas, \emph{Weyl
  asymptotics for tensor products of operators and dirichlet divisors}, Ann.
  Mat. Pura Appl. (2014), doi: 10.1007/s10231-014-0400-z.

\bibitem[GS94]{GS94}
Alain Grigis and Johannes Sj{\"o}strand, \emph{Microlocal analysis for
  differential operators. An Introduction}. In: London Mathematical Society Lecture Note Series,
  vol. 196, Cambridge University Press, Cambridge, 1994, An introduction.
  

\bibitem[Gui85]{GU85}
Victor Guillemin, \emph{A new proof of {W}eyl's formula on the asymptotic
  distribution of eigenvalues}, Adv. in Math. \textbf{55} (1985), no.~2,
  131--160. 

\bibitem[Har16]{HA16}
G.~H. Hardy, \emph{On {D}irichlet's {D}ivisor {P}roblem}, Proc. London Math.
  Soc. (2) \textbf{15} (1916), 1--25.

\bibitem[H{\"o}r68]{HO68}
Lars H{\"o}rmander, \emph{The spectral function of an elliptic operator}, Acta
  Math. \textbf{121} (1968), 193--218. 

\bibitem[H{\"o}r07]{HO04}
\bysame, \emph{The analysis of linear partial differential operators. {IV}: Fourier Integral Operators}. In:
  Classics in Mathematics, Springer, Berlin, (2007), Reprint of the 1994 edition. 

  
\bibitem[He84]{He84}  
B.~Helffer, \emph{Th\'eorie spectrale pour des op\'erateurs globalement
              elliptiques}, Ast\'erisque (112), Soci\'et\'e Math\'ematique de France, Paris,
              1984.
  
\bibitem[HR81]{HL81}
B.~Helffer and D.~Robert, \emph{Comportement asymptotique pr\'ecise du spectre
  d'op\'erateurs globalement elliptiques dans {${\bf R}^{n}$}},
  Goulaouic-{M}eyer-{S}chwartz {S}eminar, 1980--1981, \'Ecole Polytech.,
  Palaiseau, 1981, pp.~Exp. No. II, 23. 

\bibitem[IKKN06]{IKKN06}
A.~Ivi{\'c}, E.~Kr{\"a}tzel, M.~K{\"u}hleitner, and W.~G. Nowak, \emph{Lattice
  points in large regions and related arithmetic functions: recent developments
  in a very classic topic}, Elementare und analytische {Z}ahlentheorie, Schr.
  Wiss. Ges. Johann Wolfgang Goethe Univ. Frankfurt am Main, 20, Franz Steiner
  Verlag Stuttgart, Stuttgart, 2006, pp.~89--128. 

\bibitem[Mor08]{MO08}
Sergiu Moroianu, \emph{Weyl laws on open manifolds}, Math. Ann. \textbf{340}
  (2008), no.~1, 1--21. 

\bibitem[Nic03]{NI03}
Fabio Nicola, \emph{Trace functionals for a class of pseudo-differential
  operators in {$\mathbb R\sp n$}}, Math. Phys. Anal. Geom. \textbf{6} (2003),
  no.~1, 89--105. 

\bibitem[NR06]{NR06}
Fabio Nicola and Luigi Rodino, \emph{Residues and index for bisingular
  operators}, {$C^\ast$}-algebras and elliptic theory, Trends Math.,
  Birkh\"auser, Basel, 2006, pp.~187--202.

\bibitem[Rod75]{RO75}
Luigi Rodino, \emph{A class of pseudo differential operators on the product of
  two manifolds and applications}, Ann. Scuola Norm. Sup. Pisa Cl. Sci. (4)
  \textbf{2} (1975), no.~2, 287--302.

\bibitem[See67]{SE67}
R.~T. Seeley, \emph{Complex powers of an elliptic operator}, Singular
  {I}ntegrals ({P}roc. {S}ympos. {P}ure {M}ath., {C}hicago, {I}ll., 1966),
  Amer. Math. Soc., Providence, R.I., 1967, pp.~288--307. 

\bibitem[Shu87]{SH87}
M.~A. Shubin, \emph{Pseudodifferential operators and spectral theory}, Springer
  Series in Soviet Mathematics, Springer-Verlag, Berlin, 1987, Translated from
  the Russian by Stig I. Andersson. 

\bibitem[SV97]{SV97}
Yu. Safarov and D.~Vassiliev, \emph{The asymptotic distribution of eigenvalues
  of partial differential operators}. In: Translations of Mathematical Monographs,
  vol. 155, American Mathematical Society, Providence, RI, 1997, Translated
  from the Russian manuscript by the authors. 

\end{thebibliography}
%\bibliographystyle{amsalpha}

\providecommand{\bysame}{\leavevmode\hbox to3em{\hrulefill}\thinspace}
\providecommand{\MR}{\relax\ifhmode\unskip\space\fi MR }
\providecommand{\MRhref}[2]{%
  \href{http://www.ams.org/mathscinet-getitem?mr=#1}{#2}
}
\providecommand{\href}[2]{#2}

\end{document}